\documentclass[10pt]{article}
\usepackage{geometry}
\usepackage{hyperref,version}
\usepackage[T1]{fontenc}
\usepackage[utf8]{inputenc}
\usepackage[american]{babel}
\usepackage{amssymb}
\usepackage{amsmath,amsthm}
\numberwithin{equation}{section}
\usepackage{graphicx}
\usepackage{placeins}
\usepackage{xspace}
\usepackage{bm,bbm}
\usepackage{centernot}
\usepackage[numbers]{natbib}
\usepackage{enumerate,array}
\usepackage{xcolor}
\usepackage{dsfont}
\usepackage{enumitem}
\setlist{leftmargin=20pt}
\usepackage{Macros/mydef}
\usepackage{multirow}
\usepackage{diagbox}
\usepackage{lineno}
\usepackage{soul}
\usepackage{fullpage}

\usepackage{algorithm}
\usepackage{algpseudocode}
\usepackage{subcaption}
\usepackage{comment}
\usepackage{cancel}

\begin{document}
\newcommand\footnotemarkfromtitle[1]{%
\renewcommand{\thefootnote}{\fnsymbol{footnote}}%
\footnotemark[#1]%
\renewcommand{\thefootnote}{\arabic{footnote}}}

\title{A high-order nodally bound-preserving and mass-conservative method for linear fourth-order elliptic problems and its applications to nonlinear parabolic equations }

\author{Jie Shen\footnotemark[1], Zuodong Wang\footnotemark[1]}%

\date{\today}

\maketitle

\renewcommand{\thefootnote}{\fnsymbol{footnote}}

\footnotetext[1]{School of Mathematical Sciences; Eastern Institute of Technology, Ningbo, Zhejiang, 315200, China}%

\renewcommand{\thefootnote}{\arabic{footnote}}
\begin{abstract}
We propose a high-order finite element method for linear fourth-order elliptic problems that is both nodally bound-preserving and mass-conservative, based on a variational inequality formulation. The method admits an equivalent strictly convex minimization structure, which ensures well-posedness and enables an optimal error estimate in the $H^1$-seminorm, under suitable regularity assumptions.

This framework is further extended to nonlinear fourth-order parabolic problems through space--time high-order discretizations that combine variational inequalities, BDF schemes, and scalar auxiliary variable (SAV) techniques. The fully discrete schemes preserve nodal bounds and mass, and a modified energy stability result is established for the first-order temporal scheme. We also apply the same framework to nonlinear second-order parabolic problems by introducing a consistent fourth-order regularization, leading to space--time high-order schemes with the same bound-preserving and mass-conservative properties.

Extensive numerical results, including challenging tests with singularities and low regularity, demonstrate the stability, efficiency, and high-order accuracy of the proposed methods.



\end{abstract}


\section{Introduction}
Fourth-order partial differential equations arise in a wide range of applications in physics and engineering, especially in materials science and fluid dynamics \cite{Anderson_McFadden_Wheeler_1998,Cahn_Hilliard_1958,Elder_Katakowski_Haataja_Grant_2002,Gurtin_Polignone_Vinals_1996,Leslie_1979,Yue_Feng_Liu_Shen_2004}. These models often exhibit intrinsic structural properties such as pointwise bound-preservation, mass-conservation, and energy dissipation, which are essential for physical fidelity. Designing high-order numerical schemes that preserve these structural properties remains a significant challenge.

In this work, we consider the following prototypical fourth-order elliptic system posed on a Lipschitz domain $\Dom\subset\Real^d$, $d\in\{1,2,3\}$:
\begin{equation}\label{eq_model_stationary}
    u = \DIV(M \nabla w ) + f_1, \qquad w=-\DIV(\kappa \nabla u) + f_2,
\end{equation}
where $f_1\in L^2(\Dom)$ with $\frac{1}{|\Dom|}\int_\Dom f_1 \in [a,b]$ and $-\infty<a<b<\infty$, 
$f_2\in L^2(\Dom)$ with $\int_\Dom f_2=0$\footnote{$\frac{1}{|\Dom|}\int_\Dom f_1 \in [a,b]$ is a necessary condition for $u$ to be bound-preserving. The mean-free constraint is used only in the error analysis. This does not affect $u$ due to invariance under constant shifts in $w$ with the boundary conditions considered.}, and $M,\kappa\in L^\infty(\Dom;\Real^+)$ are uniformly bounded from below and above. Homogeneous Neumann boundary conditions are imposed for both $u$ and $w$. 
A fundamental property of the solution in \eqref{eq_model_stationary} are the mass-conservation
\begin{equation}
    \int_\Dom u =\int_\Dom f_1,
\end{equation}
which follows from the first equation in \eqref{eq_model_stationary} and the homogeneous Neumann boundary condition. In applications considered here, admissible solutions are
required to satisfy the bound constraint
\begin{equation}
u(\xcoord)\in[a,b] \quad \text{a.e. in } \Dom.
\label{eq:mass-conservation}
\end{equation}
We further replace
the second equation in \eqref{eq_model_stationary} by a variational inequality over the closed convex set of
admissible functions. If a sufficiently regular solution of the unconstrained
system is known a priori to satisfy \(a\le u(\xcoord)\le b\) and the corresponding
complementarity condition, then it is also a solution of the variational
inequality formulation.

Mass-conservation and bound-preservation are often required at the discrete level, to ensure physical admissibility and stability, specifically for time-semidiscretized PDEs; see, e.g., \cite{Cahn_Hilliard_1958, Bertozzi_Pugh_1996, Zhornitskaya_Bertozzi_2000,Shen_Xu_Yang_2019}.

In addition to \eqref{eq_model_stationary}, we consider two classes of nonlinear parabolic problems. The first class consists of fourth-order problems:
\begin{equation}
\partial_t u = \nabla\cdot(M\nabla w),\quad
w = -\Delta u + f(u),
\label{eq:ch-general}
\end{equation}
including the special case
\begin{equation}
\partial_t u = \nabla\cdot(M\nabla w),\quad
w = -\Delta u,
\label{eq:ch-linear}
\end{equation}
with $u_0\in H^1(\Dom;[a,b])$. Here $f\in \calC^0([a,b];\Real)$ is the derivative of a potential $F\in \calC^1([a,b];\Real)$, and $M$ is uniformly positive. Solutions typically satisfy bound-preservation, mass-conservation, and energy dissipation $E(u(t))\le E(u_0)$, where
\begin{equation}
E(v) \eqq (\frac12\|\nabla v\|^2) + (\int_D F(v)+C_0)\eqq E_0(v) + E_1(v),
\label{eq:energy}
\end{equation}
with $C_0>0$ a constant ensuring the positivity of $E_1$.
Typical examples include Cahn--Hilliard models \cite{Cahn_Hilliard_1958} and lubrication-type equations \cite{Zhornitskaya_Bertozzi_2000}.
We also consider nonlinear second-order parabolic problems
\begin{equation}
\partial_t u = \nabla\cdot(K\nabla u),
\label{eq:second-order-parabolic}
\end{equation}
with $u_0\in L^2(\Dom;[a,b])$ and $K\in L^\infty([a,b]\times \Dom\times [0,T];\Real^+)$ uniformly bounded from below and above.
Their solutions satisfy bound-preservation, mass-conservation, and energy dissipation. Typical examples are porous medium equations \cite{Vazquez_2007}.

Over the past decades, many numerical schemes have been developed to preserve bounds and mass for second-order elliptic problems within the finite element framework \cite{Ern_Guermond_FEM1_2021}. In low-order settings, discrete maximum principles or bound-preserving properties can be achieved under suitable mesh conditions \cite{Barrenechea_John_Knobloch_2024}. Variational inequality formulations enable high-order bound-preserving schemes \cite{Barrenechea_Georgoulis_Pryer_2023,Kirby_Shapero_2024}, together with efficient solvers such as primal--dual active set methods \cite{Hintermuller_Ito_Kunisch_2002}. Optimal convergence results have recently been established and extended to other model problems and spatial discretizations in \cite{Barrenechea_Georgoulis_Pryer_2023,Amiri_Barrenechea_Pryer_2024,Amiri_Barrenechea_Pryer_2025,Dong_Ern_Wang_2025,Barrenechea_Pryer_Trenam_2025}. However, these approaches generally do not ensure mass-conservation. This limitation has been addressed in \cite{Barrenechea_Philip_Andreas_2025} via low-order enriched Galerkin methods, and by postprocessing techniques such as optimization-based approaches \cite{Liu_Riviere_Shen_Zhang_2024} and mass redistribution \cite{Guermond_Wang_2025}.
However, for fourth-order elliptic problems, high-order methods that simultaneously preserve bounds and mass remain scarce. \cite{Blank_Butz_Garcke_2011} and references therein apply the lowest-order variational inequality techniques to ensure the bound-preservation and mass-conservation, but rigorous convergence analysis is not given. For simplicity, the term \emph{bound-preservation} will hereafter refer to nodal bound-preservation in the finite element setting.

For parabolic problems, structure-preserving schemes have been extensively studied in the time-semidiscrete setting, especially for energy stability. Representative approaches include convex splitting \cite{Elliott_Stuart_1993}, linear stabilization \cite{Shen_Yang_2010}, invariant energy quadratization \cite{Yang_2016}, scalar auxiliary variable (SAV) methods \cite{Shen_Xu_Yang_2019}, generalized SAV methods \cite{Huang_Shen_2022}, and Lagrange multiplier techniques \cite{Cheng_Liu_Shen_2020}. While these methods ensure energy stability, they generally do not enforce bounds or mass-conservation simultaneously. Lagrange multiplier approaches \cite{Cheng_Shen_2022_1,Cheng_Shen_2022} preserve bounds and mass but may lose energy stability, whereas other methods preserve bounds and energy stability \cite{Cheng_Wang_Zhao_2025,Huang_Shen_2021}.

To the best of our knowledge, for fourth-order elliptic problems, there is no existing high-order method that preserves bounds and mass; for nonlinear parabolic problems, no existing method simultaneously preserves bounds, mass, and energy dissipation in the high-order FEM setting.
Motivated by these gaps, we develop a variational inequality framework for elliptic and parabolic model problems introduced above. For linear fourth-order elliptic equations, we design a high-order finite element scheme that preserves bounds and mass, and analyze the proposed scheme by considering the associated convex minimization problem posed in $H^{-1}$-$H^1$-norms. For nonlinear fourth-order parabolic problems, we combine the elliptic framework with BDF temporal discretizations and SAV techniques to construct high-order schemes that preserve bounds and mass, and a modified energy stability is satisfied for first-order temporal discretization and arbitrarily high-order spatial discretization. We further extend the approach to nonlinear second-order problems through a consistent fourth-order regularization.

The main contributions of this work are summarized as follows:
\begin{itemize}
  \item \textbf{Linear fourth-order elliptic problems.} We propose high-order variational inequality formulations that yield finite element schemes preserving both bounds and mass. The discrete problem admits an equivalent strictly convex minimization formulation in the $H^{-1}$-$H^1$-norms, which ensures well-posedness and enables the derivation of an optimal Céa-type $H^1$-error estimate. The optimal convergence rate is later derived by constructing a bound-preserving and mass-conservative counterpart of the quasi-interpolation proposed in \cite{Ern_Guermond_2017}.

  \item \textbf{Nonlinear fourth-order parabolic problems.} We develop space--time high-order schemes by combining the variational inequality framework with BDF temporal discretizations and SAV techniques. The resulting methods preserve bounds, mass. The well-posedness is established through the corresponding convex minimization problem. For first-order temporal discretization, a modified energy stability is also established.

  \item \textbf{Nonlinear second-order parabolic problems.} We introduce a consistent fourth-order regularization strategy that allows second-order problems to be treated within the fourth-order parabolic framework, leading to schemes that preserve bounds and mass. While a complete well-posedness theory is not yet available, the stability and accuracy of the approach are strongly supported by numerical evidence.
\end{itemize}

The paper is organized as follows. Section~\ref{Sec:linear_elliptic} presents the elliptic scheme and its analysis. Section~\ref{Sec:application_parabolic} develops the parabolic schemes. Numerical results are given in Section~\ref{sec:numerics}, followed by concluding remarks in Section~\ref{sec:conclusions}.

Throughout the paper, we use standard notation for Sobolev spaces. The dual of $H^1(\Dom)$ is denoted by $H^{-1}(\Dom)$, and the mean-free subspace of $L^2(\Dom)$ by $L^2_0(\Dom)$. The $L^2(\Dom)$-inner product is written $(\cdot,\cdot)$ with associated norm $\norm{\cdot}$; we omit explicit space--time dependencies when no confusion can arise (e.g. $\|\cdot\|_{L^2(L^2)}:=\|\cdot\|_{L^2((0,T];L^2(\Dom))}$). Generic constants $C>0$ may change from line to line but are always independent of the mesh size and time step. For brevity, spatial and temporal dependencies of coefficients are often suppressed when clear from context (e.g. $M(u)=M(u,x,t)$, $\kappa=\kappa(x,t)$).

\section{A mass-conservative and bound-preserving scheme for linear fourth-order elliptic problems}\label{Sec:linear_elliptic}
In this section, we develop a high-order finite element discretization of \eqref{eq_model_stationary} based on a variational inequality formulation. The main objectives are to enforce bound-preservation and mass-conservation at the discrete level, and to establish well-posedness and an optimal error estimate in the $H^1$-seminorm.

\subsection{Spatial discrete setting and numerical scheme}

For the spatial discretization, we consider either a simplicial or a tensor-product conforming mesh $\{\calT_h\}_{h>0}$ of $\Dom$, with mesh size parameter $h\eqq \max_{K\in\calT_h} \text{diam}(K)$. 
A generic mesh cell is denoted by $K\in \calT_h$. We assume only that the mesh is shape-regular, rather than the \emph{acute angle} assumption which guarantees the discrete maximum principle in the literature (see for example \cite{Barrenechea_John_Knobloch_2024}). We focus on continuous Lagrange finite elements of order $p\geq 1$. Thus, each degree of freedom (dof) $i$ corresponds to 
the value at the mesh node $\xcoord_i$ for all $i\in\DoFSet\eqq \intset{1}{I}$, where $I$ is the total number of mesh nodes. The global shape functions are
denoted by $\{\varphi_i\}_{i \in \DoFSet}$. The corresponding $H^1$-conforming FEM space is denoted by 
\begin{equation}
    V_h^p\eqq\text{span}\{\varphi_i\}_{i \in \DoFSet}.
\end{equation}
We define $V_h^{p,+}\subset V_h^p$ as the subset of bound-preserving functions (in the nodal sense) in $V_h^p$, \ie
\begin{equation}
    V_h^{p,+}\eqq \Big\{ v_h\eqq \sum_{i\in\DoFSet}V_i\varphi_i \in V_h^p \: |\:  V_i \in [a,b], \ \forall i \in \DoFSet \Big\}.
\end{equation}

We now construct a high-order,  bound-preserving and mass-conservative numerical scheme for the problem \eqref{eq_model_stationary}. The key idea is to reformulate the PDE as a variational inequality, which naturally incorporates the bound constraint through a convex admissible set. 
Following the argument in \cite[Chapter 3]{Han_2024} and \cite{Blank_Butz_Garcke_2011}, the problem \eqref{eq_model_stationary} with a sufficiently regular solution can be reformulated as finding $(u,w)\in H^1(\Dom)\cap H^2(\Dom;[a,b])\times H^1(\Dom)$ such that
\begin{equation}\label{eq_model_stationary_VI}
    \begin{cases}
        (u, v) + (M\nabla w , \nabla v)  = (f_1,v), \quad \forall v\in H^1(\Dom),\\
        (w,\xi-u) \leq (\kappa\nabla u, \nabla ( \xi-u) ) +(f_2, \xi-u), \quad \forall \xi\in H^1(\Dom;[a,b]).
    \end{cases}
\end{equation}
This inequality can be interpreted as enforcing the bound constraint via projection onto a convex admissible set, leading naturally to a variational inequality formulation.
We then discretize the continuous problem by replacing trial and test functions from $H^1(\Dom)$ and $H^1(\Dom;[a,b])$ to $V_h^p$ and $V_h^{p,+}$, respectively, to get the numerical scheme: find $(u_h,w_h)\in V_h^{p,+} \times V_h^p$ such that
\begin{equation}\label{eq_scheme_stationary}
    \begin{cases}
        (u_h, v_h) + (M\nabla w_h , \nabla v_h) = (f_1,v_h), \quad \forall v_h\in V_h^p,\\
        (w_h,\xi_h-u_h) \leq (\kappa\nabla u_h, \nabla ( \xi_h-u_h) ) +(f_2, \xi_h-u_h), \quad \forall \xi_h\in V_h^{p,+}.
    \end{cases}
\end{equation}

\subsection{Fundamental properties}
In this section, we establish the well-posedness,  bound-preservation and mass-conservation of the proposed scheme \eqref{eq_scheme_stationary}, by considering its associated minimization problem posed in the $H^{-1}$-$H^1$-norms.

We start with notation which will be frequently used throughout this and the next section.  
The weighted Poisson solver $(-\Delta^M)^{-1}:L^2_0(\Dom) \to H^1(\Dom)$, $g \mapsto z\eqq z(g)$ is given by
\begin{equation}\label{eq_weighted_Poisson_stationary}
    (M\nabla z, \nabla v) = (g,v),\qquad \int_\Dom z=0, \qquad \forall v\in H^1(\Dom).
\end{equation}
Based on this operator, we have an integration by parts property in the $H^{-1}$-inner product. We  denote $z_{v}\eqq (-\Delta^M)^{-1}v$ and $z_{w}\eqq (-\Delta^M)^{-1}w$ as the solution of \eqref{eq_weighted_Poisson_stationary} with source $v\in L^2_0(\Dom)$ and $w\in L^2_0(\Dom)$, respectively. Then, the following identity holds true:
\begin{align}
    &\quad (w,(-\Delta^M)^{-1} v) = (M\nabla z_{w}, \nabla (-\Delta^M)^{-1} v  ) = (  M \nabla (-\Delta^M)^{-1} v , \nabla (-\Delta^M)^{-1} w  )\nonumber\\
    & = (M\nabla z_{v}, \nabla (-\Delta^M)^{-1} w  ) = (v,(-\Delta^M)^{-1} w).
\end{align}
Following the argument in the continuous setting above, the weighted discrete Poisson solver $(-\Delta_h^M)^{-1}: L^2_0(\Dom) \to V_h^p$, $g \mapsto z_h$ is defined by
\begin{equation}\label{eq_weighted_Poisson_stationary_discrete}
    (M \nabla z_h, \nabla v_h) = (g,v_h), \qquad \int_\Dom z_h =0, \qquad \forall v_h \in V_h^p.
\end{equation}
We also have an integration by parts property in the discrete $H^{-1}$-inner product: By denoting $z_{h,v}\eqq (-\Delta_h^M)^{-1}v$ and $z_{h,w}\eqq (-\Delta_h^M)^{-1}w$ as the solution of \eqref{eq_weighted_Poisson_stationary_discrete} with source $v\in L^2_0(\Dom)$ and $w\in L^2_0(\Dom)$, we have
\begin{align}
    (w,(-\Delta_h^M)^{-1} v) = (  M \nabla (-\Delta_h^M)^{-1} v , \nabla (-\Delta_h^M)^{-1} w  ) = (v,(-\Delta_h^M)^{-1} w).\label{eq_H_minus_1_IBP}
\end{align}
In addition, we introduce the weighted discrete $H^{-1}$-seminorm for all $v\in L^2_0(\Dom)$ by
\begin{equation}
\norm{v}_{H^{-1}_h}^2\eqq (  M \nabla (-\Delta_h^M)^{-1} v , \nabla (-\Delta_h^M)^{-1} v  ).
\end{equation}

We are now in a position to prove the well-posedness, mass-conservation and  bound-preservation  for the proposed scheme \eqref{eq_scheme_stationary}. 
\begin{lemma}[Well-posedness for the linear elliptic scheme \eqref{eq_scheme_stationary}]\label{lemma:well_posedness_stationary}
    The variational inequality \eqref{eq_scheme_stationary} admits a solution $(u_h,w_h)$. The component
    $u_h$ is unique, bound-preserving, and mass-conservative. Moreover, $u_h$ is the unique minimizer of the following problem:
    \begin{equation}\label{eq_minimization_dis_stationary}
        \min_{v \in Y  } \Big\{  \frac{1}{2}\norm{\sqrt{\kappa}\nabla v}^2 + (f_2,v) + \frac{1}{2} \norm{v - f_1}_{H^{-1}_h}^2 \Big\},
    \end{equation}
    where
    \begin{equation}
        Y=\{v\in V_h^{p,+} \ | \  \int_\Dom v = \int_\Dom f_1\}.
    \end{equation}
    
    In addition, the solution $u_h$ satisfies the following variational inequality for all $\xi_h\in Y$:
    \begin{equation}\label{eq_num_VI_primal_stationary}
        ((-\Delta_h^M)^{-1} ( u_h - f_1) , \xi_h-u_h ) + (\kappa \nabla u_h, \nabla (\xi_h -u_h) ) + (f_2, \xi_h-u_h) \geq 0.
    \end{equation}
\end{lemma}

\begin{remark}[$H^{-1}$-$H^1$ pair]
The appearance of the discrete $H^{-1}$ and $H^1$ inner products in \eqref{eq_num_VI_primal_stationary} and \eqref{eq_minimization_dis_stationary} allows the fourth-order problem to be reformulated as a strictly convex minimization problem without $w$. This reformulation avoids terms containing $w-w_h$ in the later error analysis in Theorem \ref{Theorem:error_estimate}. We note that error is hard to control in the original variational inequality \eqref{eq_scheme_stationary}, owing to the inequality constraint and the lack of the Poincaré inequality for $w-w_h$.
\end{remark}

\begin{proof}
    We first establish the well-posedness of the strictly convex minimization problem \eqref{eq_minimization_dis_stationary}. We then show that its associated KKT system is equivalent to the variational inequality \eqref{eq_scheme_stationary}, which yields the well-posedness of \eqref{eq_scheme_stationary}. Finally, we derive \eqref{eq_num_VI_primal_stationary} as the Euler inequality associated with \eqref{eq_minimization_dis_stationary}.
    
    \paragraph*{Step 1: well-posedness of the minimization problem.} We show first that the admissible set $Y$ is nonempty, convex and closed.
    
    \textbf{Nonemptiness.} Since the mean value of $f_1$ is well-defined and assumed to lie in $[a,b]$, the constant function $\frac{1}{|\Dom|}\int_\Dom f_1$ belongs to $V_h^{p,+}$ and satisfies the mass constraint, implying that $Y$ is nonempty.

    \textbf{Convexity.} For any $\theta\in [0,1]$, and any pairs $v_1$ and $v_2$ in $Y$, we have $\theta v_1+(1-\theta)v_2 \in V_h^{p,+}$ and  $\int_\Dom \theta v_1 +(1-\theta)v_2 =\int_\Dom f_1$. Hence all convex combinations of $v_1$ and $v_2$ are also in $Y$, which implies that $Y$ is convex.

    \textbf{Closedness.} Since $V_h^p$ is finite-dimensional, and constraints $V_i \in [a,b]$ are closed in $\Real^I$, $V_h^{p,+}$ is closed.

    Therefore, the minimization problem \eqref{eq_minimization_dis_stationary} is a strictly convex functional subject to a finite dimensional, nonempty, convex and closed feasible set. This implies existence and uniqueness of the minimizer for \eqref{eq_minimization_dis_stationary}.

    \paragraph*{Step 2: well-posedness of the variational inequality \eqref{eq_scheme_stationary}.} Finally, we follow the argument in \cite{Butz_2012} to prove that \eqref{eq_scheme_stationary} admits a unique solution.
    The minimization problem \eqref{eq_minimization_dis_stationary} can be reformulated as the following minimization problem:
    \begin{align*}
        \qquad&\min_{u_h,v_h \in V_h^p} \Big\{  \frac{1}{2}\norm{\sqrt{\kappa}\nabla u_h}^2 + (f_2,u_h) +\frac{1}{2} (M\nabla v_h, \nabla v_h) \Big\}\\
        &s.t.\\
        &\qquad (u_h -f_1,\chi_h)+(M\nabla v_h,\nabla \chi_h) = 0, \quad \forall \chi \in V_h^p,\quad \int_\Dom v_h = 0,\\
        &\qquad a\leq U_i\leq b,\quad\forall i\in \DoFSet, \qquad u_h=\sum_{i\in\DoFSet} U_i\varphi_i.
    \end{align*}
    The corresponding Lagrangian is given by
    \small{
    \begin{equation*}
        \frac{1}{2}\norm{\sqrt{\kappa}\nabla u_h}^2 + (f_2,u_h) +\frac{1}{2} (M\nabla v_h, \nabla v_h) - \Big( (M\nabla v_h,\nabla w_h)
        +(u_h-f_1, w_h) \Big)
        -\kappa_1 \int_\Dom v_h - \langle \mu_+, b-u_h  \rangle - \langle \mu_-, u_h-a  \rangle,
    \end{equation*}
    } 
    where we introduced the Lagrange multipliers $w_h\in V_h^p$ for the first FEM constraint, $\kappa_1\in\Real$ for the mean-free constraint of $v_h$, and $\mu_+=\sum_{i\in\DoFSet}\mu_{+,i}\varphi_i\in V_h^p$, $\mu_-=\sum_{i\in\DoFSet}\mu_{-,i}\varphi_i\in V_h^p$ for the  bound-preserving constraints on $u_h$, respectively. The inner product $\langle \cdot,\cdot \rangle$ is defined as the Euclidean inner product for dofs. 
    The associated KKT-system is given by
    \begin{align*}
        &(M\nabla (w_h-v_h),\nabla \chi_h) =(\kappa_1,\chi_h), \quad\forall \chi_h\in V_h^p,\quad \int_\Dom v_h=0,\\
        &(u_h-f_1,\chi_h) + (M\nabla v_h, \nabla \chi_h)= 0, \quad\forall \chi_h\in V_h^p,\\
        &(w_h,\chi_h)= (\kappa\nabla u_h,\nabla \chi_h) + (f_2,\chi_h) + \langle \mu_+-\mu_- , \chi_h \rangle, \quad\forall \chi_h\in V_h^p,\\
        &\mu_{+,i},\mu_{-,i}\geq 0, \mu_{+,i}(U_i-b)= 0,  \mu_{-,i}(U_i-a)= 0, \quad \forall i \in\DoFSet.
    \end{align*}
    
    Additionally, taking $\chi_h=1$ in the first equation, we obtain $\kappa_1=0$. Thus this identity with the first and the second equations in the above KKT system implies that $v_h$ and $w_h$ only differ by a constant. Therefore, replacing $\nabla v_h$ by $\nabla w_h$, we infer a simplified KKT-system
    \begin{align*}
        &(u_h-f_1,\chi_h) + (M\nabla w_h, \nabla \chi_h)= 0, \quad\forall \chi_h\in V_h^p,\\
        &(w_h,\chi_h)= (\kappa\nabla u_h,\nabla \chi_h) + (f_2,\chi_h) + \langle \mu_+-\mu_- , \chi_h \rangle, \quad\forall \chi_h\in V_h^p,\\
        &\mu_{+,i},\mu_{-,i}\geq 0, \mu_{+,i}(U_i-b)= 0,  \mu_{-,i}(U_i-a)= 0, \quad \forall i \in\DoFSet.
    \end{align*}
    We now show that the second and the third line in the simplified KKT-system is equivalent to
    \begin{equation*}
        (w_h,\xi_h-u_h) \leq (\kappa\nabla u_h,\nabla \xi_h-u_h) + (f_2,\xi_h-u_h) \quad\forall \xi_h\in V_h^{p,+}.
    \end{equation*}
    We prove first that the second and the third line in the simplified KKT-system implies the variational inequality above.
    Indeed, for all $\xi_h\eqq \sum_{i\in\DoFSet} \xi_i \varphi_i \in V_h^{p,+}$, we have 
    \begin{equation*}
        (\mu_{+,i}-\mu_{-,i})(\xi_i-U_i)\leq \mu_{+,i}(b-U_i) - \mu_{-,i}(a-U_i) = 0, \qquad \forall i\in \DoFSet,
    \end{equation*}
    which implies $\langle \mu_+-\mu_- , \xi_h -u_h \rangle\leq 0$. Setting $\chi_h=\xi_h-u_h$ in the second line of the simplified KKT-system and using $\langle \mu_+-\mu_- , \xi_h -u_h \rangle\leq 0$, we obtain the expected implication. 
    To establish the converse implication, we construct Lagrange multipliers $\mu_+$ and $\mu_-$ at the level of dofs. For each $i\in\DoFSet$, we consider admissible perturbations of the form $\xi_h = u_h + \delta \varphi_i\in V_h^{p,+}$, where the sign of $\delta$ is restricted by the bound-preserving constraints and $|\delta|$ is sufficiently small. By testing the variational inequality with such perturbations, we obtain the following dofwise definition:
    \[
    \begin{cases}
    \mu_{+,i}=0,\quad \mu_{-,i}=(\kappa\nabla u_h,\nabla\phi_i)+(f_2,\phi_i)-(w_h,\phi_i), & \text{if } U_i=a,\\[0.4em]
    \mu_{-,i}=0,\quad \mu_{+,i}=(w_h,\phi_i)-(\kappa\nabla u_h,\nabla\phi_i)-(f_2,\phi_i), & \text{if } U_i=b,\\[0.4em]
    \mu_{+,i}=\mu_{-,i}=0, & \text{otherwise.}
    \end{cases}
    \]
    This definition satisfies the complementarity conditions
    \[
    \mu_{+,i},\mu_{-,i} \ge 0,\quad \mu_{+,i}(U_i-b)=0,\quad \mu_{-,i}(U_i-a)=0,
    \]
    together with
    \[
    (w_h,\varphi_i) = (\kappa \nabla u_h,\nabla \varphi_i) + (f_2,\varphi_i) + (\mu_{+,i}-\mu_{-,i}).
    \]
    Summing over all dofs yields the expected representation
    \begin{equation*}
        (w_h,\chi_h)= (\kappa\nabla u_h,\nabla \chi_h) + (f_2,\chi_h) + \langle \mu_+-\mu_- , \chi_h \rangle, \quad\forall \chi_h\in V_h^p,
    \end{equation*}
    and hence the equivalence between the variational inequality and the simplified KKT system. Therefore, the existence and uniqueness of the minimization problem leads to the well-posedness of \eqref{eq_scheme_stationary}.

    \paragraph*{Step 3: Euler inequality for the minimization problem.}
    We apply \eqref{eq_H_minus_1_IBP} to reformulate the minimization problem \eqref{eq_minimization_dis_stationary} as 
    \begin{equation*}
        \min_{v\in Y} \calJ(v),\qquad  \calJ(v)= \frac{1}{2}\norm{\sqrt{\kappa}\nabla v}^2 +(f_2,v) + \frac{1}{2}((-\Delta_h^M)^{-1}(v-f_1),v-f_1).
    \end{equation*}
    Its minimizer $u_h$ satisfies the first-order optimality condition
    \begin{equation*}
        (\frac{\mathrm{d}}{\mathrm{d}v}\calJ(u_h),\xi_h-u_h) \geq 0, \qquad \forall \xi_h\in Y.
    \end{equation*}
    We now explore the explicit formula of $\frac{\mathrm{d}}{\mathrm{d}v}\calJ(u_h)$ in all admissible directions. In a neighborhood of the minimizer $u_h\in Y$, we take any perturbation $u_h+\delta v_h\in Y$, with $\delta\in\Real\setminus\{0\}$ and $v_h\in V_h^p$. Since $\int_\Dom u_h = \int_\Dom u_h+\delta v_h = \int_\Dom f_1$, we must have $\int_\Dom v_h =0$. Therefore, this perturbation with \eqref{eq_H_minus_1_IBP} implies
    \begin{align*}
        \calJ(u_h+\delta v_h) =\calJ(u_h) + \delta \Big( ((-\Delta_h^M)^{-1}(u_h-f_1) + f_2, v_h ) + (\kappa\nabla u_h,\nabla v_h) \Big)+ O(\delta^2).
    \end{align*}
    Since the perturbation is arbitrary, we find from the above calculation the expected variational inequality \eqref{eq_num_VI_primal_stationary}.
    This concludes the proof.
\end{proof}

\begin{remark}[Stampacchia’s theorem]
    For the second-order elliptic variational inequalities, a standard way to establish the well-posedness of the variational inequality is to apply Stampacchia’s Theorem, as in \cite{Barrenechea_Pryer_Trenam_2025}. However, due to the lack of coercivity of the associated operator, Stampacchia’s theorem cannot be directly applied in the present setting, and we bypass this problem by using the minimization problem \eqref{eq_minimization_dis_stationary}.
\end{remark}

\subsection{Error estimate}
In this section, we derive an optimal error estimate in the $H^1$-seminorm, for sufficiently smooth solutions.

In order to derive an optimal error bound of the proposed scheme, we require a mass-conservative and  bound-preserving variant of the quasi-interpolation used in \cite{Ern_Guermond_2017}.
\begin{lemma}[A bound-preserving and mass-conservative map]\label{lemma:BP_MC_interpolation}
    Let $v\in\calC^0(\Bar{\Dom};[a,b])$. Assume that $h < h_0(v)$ with $h_0(v)$ a constant depending on $v$. Then, there is a bound-preserving and mass-conservative $\Pi_h v\in V_h^{p,+}$, satisfying the following estimate
    \begin{equation}
        \norm{\Pi_h v - v}_{L^\infty} \leq C \inf_{\chi_h\in V_h^p} \norm{\chi_h-v}_{L^\infty}.
    \end{equation}
    In addition, for $v\in W^{p,\infty}(\Dom)\cap H^{p+1}(\Dom)$, we have
    \begin{equation}
        \norm{\nabla (\Pi_h v -v)} \leq C h^p.
    \end{equation}
\end{lemma}

\begin{remark}[Regularity assumption on $v$]
    In order to obtain the optimal convergence rate in the $H^1$-seminorm, we assumed $v\in W^{p,\infty}(\Dom)$ to control the error introduced by the mass-conservative and bound-preserving correction on the quasi-interpolation of $v$; $v\in H^{p+1}(\Dom)$ is used to control the standard $H^1$-approximation errors of $v$.
\end{remark}

\begin{proof} 
    We assume that $v$ is not a constant, since the constant case is trivial.  
    We first consider the case $a=-b$, and then generalize it to the case $a<b$ with $a\neq -b$.
    \paragraph*{Case $a=-b$.}
    We consider the quasi-interpolation $\calI_h v$ introduced in \cite{Ern_Guermond_2017}, which is $L^\infty$- and $H^1$-stable, and satisfies the corresponding optimal convergence properties under our regularity assumption (see, for example, \cite[Theorem 22.6 and Corollary 22.8]{Ern_Guermond_FEM1_2021}).
    Then, we construct the map
    \begin{equation*}
        \Tilde{v}_h \eqq \calI_h v + c_1 \eqq \sum_{i\in\DoFSet} \Tilde{V}_i \varphi_i, \qquad c_1\eqq \frac{1}{|\Dom|} \int_\Dom v-\calI_h v.
    \end{equation*}
    It follows that $\Tilde{v}_h$ is mass-conservative and 
    \begin{equation}\label{eq_approx_c1}
        |c_1|\leq \norm{\calI_h v -v}_{L^\infty} \leq C \inf_{\chi_h\in V_h^p} \norm{\chi_h-v}_{L^\infty}.
    \end{equation}
    Since $\norm{\Tilde{v}_h-v}_{L^\infty} \leq \norm{\calI_h v - v}_{L^\infty} + |c_1| \leq C \norm{\calI_h v -v}_{L^\infty}$ and $\norm{\nabla (\Tilde{v}_h-v)} = \norm{\nabla (\calI_h v -v)}$, $\Tilde{v}_h$ inherits the $L^\infty$- and $H^1$-approximation properties of $\calI_h v$.

    In the case $\norm{v}_{L^\infty}<b$, we have
    \begin{equation*}
        \norm{\Tilde{v}_h}_{L^\infty} \leq \norm{\Tilde{v}_h-v}_{L^\infty} + \norm{v}_{L^\infty}\leq \norm{\calI_h v -v}_{L^\infty} + |c_1| + \norm{v}_{L^\infty} \leq C \inf_{\chi_h\in V_h^p} \norm{\chi_h-v}_{L^\infty} + \norm{v}_{L^\infty}.
    \end{equation*}
    Therefore, there exists a constant $h_0(v)$ such that for all $h<h_0(v)$, we have $\norm{\Tilde{v}_h}_{L^\infty} < b$. Hence it is sufficient to set $\Pi_h v = \Tilde{v}_h$ in the case $\norm{v}_{L^\infty}<b$. 
    
    In the case $\norm{v}_{L^\infty}=b$, owing to the mass shifting $c_1$, $\Tilde{v}_h$ is no longer nodally bound-preserving, hence we need one more step to correct it.
    If $\Tilde{V}_i\in[a,b]$ for all $i\in\DoFSet$, we set $\Pi_h v\eqq \Tilde{v}_h$. Otherwise, we denote $m=\frac{1}{|\Dom|}\int_\Dom v$ as the average of $v$, and $c_2\in\Real$ a constant to be determined later. We construct $\Pi_h v$ by
    \begin{equation*}
        \Pi_h v = (1-c_2) \Tilde{v}_h + c_2 m.
    \end{equation*}
    By the construction, $\int_\Dom \Pi_h v - v =0$. Now, in order to make it bound-preserving at every nodal point, we need to carefully choose $c_2$ such that
    \begin{equation*}
        \norm{\Pi_h v}_{L^\infty} \leq (1-c_2) \norm{\Tilde{v}_h}_{L^\infty} + c_2 |m| \leq b.
    \end{equation*}
    This can be achieved by choosing $c_2\in(0,1)$ and $c_2\geq \frac{\norm{\Tilde{v}_h}_{L^\infty}-b}{\norm{\Tilde{v}_h}_{L^\infty}-|m|}$. Hence we set $c_2=\frac{\norm{\Tilde{v}_h}_{L^\infty}-b}{\norm{\Tilde{v}_h}_{L^\infty}-|m|}$. We claim that such $c_2$ is well defined and takes value in $(0,1)$ for sufficiently small $h$. 
    Indeed, we have that $\norm{v}_{L^\infty}=b>|m|$ (the case $b=|m|$ implies $|v|=b$, which is already covered by the constant case). Noticing that $\norm{\Tilde{v}_h}_{L^\infty}\to\norm{v}_{L^\infty}=b>|m|$ as $h\to 0$, there exists a constant $h_0'$ depending on $v$ such that for all $h<h_0'$, $\norm{\Tilde{v}_h}_{L^\infty}-|m|\geq \frac{b-|m|}{2}$. Hence the $L^\infty$-stability and approximation properties of the quasi-interpolation implies that for all $h<h_0'$, 
    \begin{equation}\label{eq_approx_c2}
        |c_2| \leq C \Big| \norm{\Tilde{v}_h}_{L^\infty}-b \Big| \leq C\Big| \norm{\Tilde{v}_h}_{L^\infty} - \norm{v}_{L^\infty} \Big|  \leq C \norm{\Tilde{v}_h-v}_{L^\infty} \leq C\norm{\calI_h v- v}_{L^\infty} \leq C \inf_{\chi_h\in V_h^p} \norm{\chi_h-v}_{L^\infty}.
    \end{equation}
    Therefore, there exists another constant $h_0\leq h_0'$ such that $c_2\in (0,1)$ for all $h<h_0$. 
    This implies that the nodal values satisfy $(\Pi_h v)(\xcoord_i)\in[a,b]$ for all $i\in\DoFSet$, and hence $\Pi_h v\in V_h^{p,+}$.
    
    Now, we apply the $L^\infty$-stability and approximation properties of the quasi-interpolation, \eqref{eq_approx_c1} and \eqref{eq_approx_c2} to find
    \begin{align*}
        \norm{v-\Pi_h v}_{L^\infty} &\leq \norm{v-\calI_h v}_{L^\infty} + \norm{\calI_h v-\Pi_h v}_{L^\infty} \leq \norm{v-\calI_h v}_{L^\infty} + \norm{\calI_h v-\Tilde{v}_h}_{L^\infty} + \norm{\Tilde{v}_h-\Pi_h v}_{L^\infty}\\
        &\leq \norm{v-\calI_h v}_{L^\infty} + |c_1| + |c_2|(\norm{\Tilde{v}_h}_{L^\infty} +|m| ) \leq C \norm{v-\calI_h v}_{L^\infty} \leq C \inf_{\chi_h\in V_h^p} \norm{\chi_h-v}_{L^\infty},
    \end{align*}
    which is the expected convergence property in the $L^\infty$-norm.

    For the $H^1$-seminorm, we use the $H^1$-stability, the optimal approximation property of the quasi-interpolation in the $H^1$-seminorm, the regularity assumption on $v$ and \eqref{eq_approx_c2} to find
    \begin{align*}
        \norm{\nabla(\Pi_h v -v)} &\leq \norm{\nabla (\Pi_h v - \Tilde{v}_h)} +\norm{\nabla (\Tilde{v}_h-v)}\leq |c_2| \norm{\nabla \Tilde{v}_h} + \norm{\nabla (\Tilde{v}_h-v)} \\
        &= |c_2| \norm{\nabla \calI_h v} + \norm{\nabla (\calI_h v-v)} \leq C |c_2| \norm{\nabla v} + \norm{\nabla (\calI_h v-v)} \leq C h^p.
    \end{align*}
    \paragraph*{General case $a<b$ with $a\neq-b$.} The same construction applies for $\Bar{v}=v-\frac{a+b}{2}$ with lower and upper bound $-\Bar{a}=\Bar{b}=\frac{b-a}{2}$, which leads to the optimal construction $\Pi_h \Bar{v}$. The proof is completed by setting $\Pi_h v\eqq \Pi_h \Bar{v} + \frac{a+b}{2}$.
\end{proof}

We recall that, owing to the Poincaré inequality and that $M$ is uniformly bounded from below and above, the following PDE stability estimate holds true with $g\in L^2_0(\Dom)$:
\begin{equation}\label{eq_stability_z}
    \norm{(-\Delta^M)^{-1}g}\leq C\norm{\nabla (-\Delta^M)^{-1}g} \leq C \norm{g}.
\end{equation}
Similarly, owing to $V_h^p\subset H^1(\Dom)$, the following stability estimate holds true for all $g\in L^2_0(\Dom)$:
\begin{equation}\label{eq_stability_zh}
    \norm{(-\Delta_h^M)^{-1}g}\leq C \norm{g}.
\end{equation}

Now, we are ready to prove that $u_h$ in \eqref{eq_scheme_stationary} has the optimal convergence rate in the $H^1$-seminorm by using the variational inequality \eqref{eq_num_VI_primal_stationary} and the bound-preserving and mass-conservative map in Lemma \ref{lemma:BP_MC_interpolation}.

\begin{theorem}[Error estimate]\label{Theorem:error_estimate}
Assume that the domain $\Dom$ has a sufficiently smooth boundary, data are sufficiently regular such that $u\in \calC^0(\Bar{\Dom};[a,b])\cap W^{p,\infty}(\Dom)\cap H^{p+1}(\Dom)$ and $w\in H^{p+1}(\Dom)$. Then, for $h<h_0(u)$ with $h_0(u)$ determined in Lemma \ref{lemma:BP_MC_interpolation}, we have the following estimate for the numerical solution $u_h$ delivered by \eqref{eq_scheme_stationary}:
\begin{equation}
    \norm{\nabla (u-u_h)} \leq C h^p.
\end{equation}
\end{theorem}

\begin{remark}[Assumptions in error analysis]
    The regularity assumptions on $u,w\in H^{p+1}(\Dom)$ are used to derive an optimal error bound in the $H^1$-seminorm, whereas $u\in \calC^0(\Bar{\Dom};[a,b]) \cap W^{p,\infty}(\Dom)$ and $h<h_0(u)$ is used to construct a bound-preserving and mass-conservative element in $Y$ with optimal $H^1$-convergence rate.
\end{remark}

\begin{proof}
The proof follows a standard approach for variational inequalities, combining a suitable projection with stability and consistency estimates.
The whole proof is divided into four steps. In the first step, we introduce the necessary notation and properties of $u$ in \eqref{eq_model_stationary} and $u_h$ in \eqref{eq_scheme_stationary}, and introduce the error decomposition. In the second step, we control the consistency error. In the third step, we use the stability of the numerical scheme to derive the bounds on the rest part of the error decomposition. In the final step, we combine the estimates in the previous steps to conclude the proof. We recall the $M$ and $\kappa$ are assumed to be uniformly bounded from below and above, hence there exists constants $M_1,M_2,\kappa_1,\kappa_2$ such that
\begin{equation*}
    0<M_1 \leq M \leq M_2 <\infty ,\qquad 0<\kappa_1 \leq \kappa \leq \kappa_2 < \infty.
\end{equation*}
These constants will be hidden into generic constant $C$ in this proof.

\paragraph*{Step 1.}
Following the argument in \cite{Butz_2012}, for sufficiently regular solutions, \eqref{eq_model_stationary} can be reformulated as 
\begin{equation}\label{eq_model_stationary_primal}
    ( (-\Delta^M )^{-1}( u - f_1),v) + ( \kappa \nabla u,\nabla v ) +(f_2,v) = 0,\quad \forall v\in H^1(\Dom),
\end{equation}
with the weighted Poisson solver  $(-\Delta^M)^{-1}$ defined in \eqref{eq_weighted_Poisson_stationary}. Indeed, from the first equation in \eqref{eq_model_stationary}, we obtain $ (-\Delta^M )^{-1}( u - f_1)=-w$ up to an additive constant. Then, the second equation in \eqref{eq_model_stationary} with $\int_\Dom f_2=0$ implies $\int_\Dom w=0$. Hence we conclude that $ (-\Delta^M )^{-1}( u - f_1)=-w$. Therefore, inserting $ (-\Delta^M )^{-1}( u - f_1)=-w$ in the second line of \eqref{eq_model_stationary}, we obtain \eqref{eq_model_stationary_primal}.

The numerical solution $u_h\in V_h^{p,+}$ of \eqref{eq_scheme_stationary} satisfies the variational inequality \eqref{eq_num_VI_primal_stationary}:
\begin{equation*}
    ( (-\Delta_h^M)^{-1}( u_h - f_1), \xi_h -u_h ) + (\kappa \nabla u_h, \nabla ( \xi_h-u_h) ) +(f_2, \xi_h-u_h) \geq 0, \quad \forall \xi_h\in Y,
\end{equation*}
with $Y= \{ v_h \in V_h^{p,+} \ | \ \int_\Dom v_h = \int_\Dom f_1 \}$.
By choosing a function $\Pi_h u \in Y$ (since $\int_\Dom u =\int_\Dom f_1$) approximating $u$ as in Lemma \ref{lemma:BP_MC_interpolation}, we decompose the error as
\begin{equation}\label{eq_error_decomposition_statonary}
    e \eqq u_h-u = (u_h -\Pi_h u) + (\Pi_h u -u) \eqq \eta +\zeta. 
\end{equation}
In addition, we introduce a parameter $\epsilon>0$, to be specified later.

\paragraph*{Step 2.}
We now estimate the consistency error. The weak form of the PDE \eqref{eq_model_stationary_primal} yields
\begin{equation}\label{eq_stationary_Piu}
    ( (-\Delta_h^M)^{-1}( \Pi_h u - f_1),v_h) + ( \kappa \nabla \Pi_h u,\nabla v_h ) +(f_2,v_h) = -(\calR,v_h), \quad \forall v_h\in V_h^p, 
\end{equation}
where the consistency term is defined by
\begin{align}
    (\calR,v_h) &= \Big(  [(-\Delta^M)^{-1}-(-\Delta_h^M)^{-1}] (u - f_1 ), v_h \Big) + \Big( (-\Delta_h^M)^{-1} (u-\Pi_h u) , v_h \Big)+ \Big( \kappa \nabla (u - \Pi_h u), \nabla v_h \Big)\nonumber\\
    &\eqq (\calR_1,v_h) + (\calR_2,v_h) + (\calR_3,v_h),
\end{align}
where $\calR_i$, $i=1,...,3$ are defined in the natural way. 

We now estimate $\calR_i$, $i=1,...,3$ one by one.
For $\calR_1$, we set $z_h\eqq (-\Delta_h^M)^{-1}(f_1-u)$ and $z\eqq (-\Delta^M)^{-1}( f_1-u )=w$, as in \eqref{eq_weighted_Poisson_stationary_discrete} and \eqref{eq_weighted_Poisson_stationary}, respectively, to get
\begin{equation*}
    (\calR_1 , v_h ) = (z-z_h,v_h) = (z-z_h,v_h - \frac{1}{|\Dom|}\int_\Dom v_h),
\end{equation*}
where we used that $z-z_h$ has zero mean and is therefore orthogonal to constants. Then, we use the Poincaré inequality, the optimal approximation properties of the elliptic projection \eqref{eq_weighted_Poisson_stationary_discrete} with assumed regularity on $M$ (see, for example \cite{Thomee_2006}), the Cauchy-Schwarz and Young's inequalities to infer
\begin{align*}
     (z-z_h,v_h - \frac{1}{|\Dom|}\int_\Dom v_h) & \leq  C\norm{\nabla (z-z_h)}\norm{\nabla v_h} \leq C\inf_{\chi_h\in V_h^p} \norm{\nabla (\chi_h-z)} \norm{\nabla v_h} \\
     &\leq \frac{C}{\epsilon} \inf_{\chi_h\in V_h^p} \norm{\nabla (\chi_h-z)}^2 + \frac{\epsilon}{3} \norm{\nabla v_h}^2.
\end{align*}
Moreover, invoking the regularity assumption in $w=z\in H^{p+1}(\Dom)$, the approximation theory leads to
\begin{equation*}
    |(\calR_1 , v_h )| \leq \frac{C}{\epsilon} h^{2p} + \frac{\epsilon}{3} \norm{\nabla v_h}^2.
\end{equation*}

For $\calR_2$, we recall that $(-\Delta_h^M)^{-1}$ maps mean-free functions to mean-free functions, use the stability \eqref{eq_stability_zh}, and apply the Poincaré, the Cauchy-Schwarz and Young's inequalities to get
\begin{align*}
    ( (-\Delta_h^M)^{-1}(u - \Pi_h u), v_h ) &= ( (-\Delta_h^M)^{-1} (u-\Pi_h u) , v_h - \frac{1}{|\Dom|}\int_\Dom v_h) \\
    & \leq C \norm{\nabla (u-\Pi_h u)} \norm{\nabla v_h} \leq \frac{C}{\epsilon} \norm{\nabla(u-\Pi_h u)}^2 + \frac{\epsilon}{3} \norm{\nabla v_h}^2.
\end{align*}
Then, we use the approximation properties of $\Pi_h u$ in Lemma \ref{lemma:BP_MC_interpolation} and the regularity assumption on $u$ to find
\begin{equation*}
    |(\calR_2,v_h)| \leq \frac{C}{\epsilon}h^{2p} +\frac{\epsilon}{3}\norm{\nabla v_h}^2.
\end{equation*}

For $\calR_3$, we apply the Cauchy-Schwarz and Young's inequalities to find 
\begin{equation*}
    (\kappa\nabla ( u-\Pi_h u ), \nabla v_h ) \leq \frac{C}{\epsilon} \norm{\nabla (u-\Pi_h u)}^2 +\frac{\epsilon}{3} \norm{\nabla v_h}^2.
\end{equation*}
Hence the PDE regularity assumption and the approximation properties of $\Pi_h u$ implies
\begin{equation*}
    |(\calR_3,v_h)|\leq \frac{C}{\epsilon} h^{2p} + \frac{\epsilon}{3}\norm{\nabla v_h}^2.
\end{equation*}

Therefore, collecting the above estimates on $\calR_i$, $i=1,2,3$, it follows
\begin{equation}\label{eq_stationary_error_consistency}
    |(\calR,v_h)| \leq \frac{C}{\epsilon} h^{2p} + \epsilon \norm{\nabla v_h}^2.   
\end{equation}

\paragraph*{Step 3.}
We now estimate $\eta$ in the error decomposition \eqref{eq_error_decomposition_statonary} .
We subtract \eqref{eq_stationary_Piu} (with $v_h=\xi_h-u_h \in V_h^p$) from \eqref{eq_num_VI_primal_stationary} and obtain
\begin{align*}
    ( (-\Delta_h^M)^{-1} \eta , \xi_h - u_h ) + (\kappa \nabla \eta, \nabla (\xi_h -u_h)  )  \geq (\calR, \xi_h-u_h).
\end{align*}
Picking $\xi_h=\Pi_h u \in Y$, invoking \eqref{eq_stationary_error_consistency}, the coercivity of operator $-\DIV (\kappa \nabla)$, and using the integration by parts property of the $H_h^{-1}$-norm in \eqref{eq_H_minus_1_IBP}, we have
\begin{align*}
    \norm{\eta}_{H^{-1}_h}^2 + C_{\kappa}\norm{\nabla \eta} ^2 \leq \norm{\eta}_{H^{-1}_h}^2 + \norm{\sqrt{\kappa}\nabla \eta} ^2 \leq (\calR, \eta)\leq \frac{C}{\epsilon}h^{2p}  + \epsilon \norm{\nabla \eta}^2,
\end{align*}
where $C_{\kappa}>0$ is a constant depending on $\kappa$.
Choosing $\epsilon=\frac{C_{\kappa}}{2}$, we obtain
\begin{equation}
    \norm{\nabla \eta} ^2 \leq C h^{2p},
\end{equation}
which leads to the expected bound on $\eta$.

\paragraph{Step 4.} We now apply Lemma \ref{lemma:BP_MC_interpolation} and the estimates in step 3 to find
\begin{equation*}
    \norm{\nabla (u-u_h)} \leq \norm{\nabla (u-\Pi_h u)} + \norm{\nabla (\Pi_h u-u_h)} = \norm{\nabla \eta} + \norm{\nabla \zeta} \leq C h^p,
\end{equation*}
which completes the proof.

\end{proof}

\subsection{Practical implementation} 

Variational inequalities with box constraints are typically solved using primal--dual active set methods; we refer to \cite{Hintermuller_Ito_Kunisch_2002,Butz_2012} and the references therein for details. In general, for a given fourth-order discretized variational inequality of the form 
\begin{equation*}
    \begin{cases}
        (u_h,v_h) + (M\nabla w_h,\nabla v_h) =(f_1,v_h), \qquad \forall v_h\in V_h^p,\\
        (w_h,\xi_h-u_h) \leq (\kappa\nabla u_h,\nabla (\xi_h-u_h)) + (f_2,\xi_h-u_h), \qquad \forall \xi_h\in V_h^{p,+},\\
    \end{cases}
\end{equation*}
with $f_1,f_2 \in L^2(\Dom)$, the primal--dual active set method proceeds as follows:

\begin{enumerate}
    \item Set threshold parameter $c>0$, iteration index $j=0$, initial guess $u_{h,0}\in V_h^p$, $A_0=\{i\in\DoFSet \ | \ U_{i,0}<a \text{ or } U_{i,0}>b  \}$, $I_0=\DoFSet\setminus A_0$.
    \item Set $V_{h,j}^p\eqq \text{span}\{\varphi_i\}_{i\in I_j}$.    
    \item  Solve $(u_{h,j+1},w_{h,j+1})\in V_h^p \times V_h^p$ via
    \begin{equation*}
        \begin{cases}
            (u_{h,j+1},v_h) + (M\nabla w_{h,j+1},\nabla v_h) =(f_1,v_h), \qquad \forall v_h\in V_h^p,\\
            (w_{h,j+1},\xi_h) = (\kappa\nabla u_{h,j+1},\nabla \xi_h) + (f_2,\xi_h), \qquad \forall \xi_h\in V_{h,j}^p,\\
            U_{i,j+1}=\min(b,\max(a,U_{i,j})), \qquad \forall i \in A_j.
        \end{cases}
    \end{equation*}
    \item Define $\mu_{h,j+1}\in V_h^p$ via 
    \begin{equation*}
        \begin{cases}
            (\mu_{h,j+1},\xi_h) = (w_{h,j+1},\xi_h) - (\kappa\nabla u_{h,j+1},\nabla \xi_h) - (f_2,\xi_h), \qquad \forall \xi_h\in V_{h,j}^p,\\
            \mu_{i,j+1}=0,\qquad \forall i\in I_j.
        \end{cases}
    \end{equation*}
    \item Set $A_{j+1}=\{i\in\DoFSet \ | \ U_{i,j+1}+\frac{\mu_{i,j+1}}{c}<a \text{ or } U_{i,j+1}+\frac{\mu_{i,j+1}}{c}>b  \}$, $I_{j+1}=\DoFSet\setminus A_{j+1}$.
    \item If $A_{j+1}=A_j$ stop; otherwise set $j=j+1$, go to step 2.
\end{enumerate}
In practice, we take $c=10^{-2}$, which is robust in our experiments. The initial guess is given by standard FEM approximation.

\section{Application to parabolic problems}\label{Sec:application_parabolic}
In this section, we extend the variational inequality framework to time-dependent problems. We construct bound-preserving and mass-conservative schemes for nonlinear fourth-order parabolic equations and propose a regularization-based extension to second-order problems. The former admits a rigorous analysis for the well-posedness, while for the latter we primarily provide numerical evidence. 

We consider three classes of problems: 1. fourth-order problem \eqref{eq:ch-linear} without a potential term; 2. fourth-order parabolic problem \eqref{eq:ch-general} with a general potential $f\neq 0$; 3. second-order parabolic problem \eqref{eq:second-order-parabolic}. The corresponding schemes are presented first, followed by their fundamental properties.

The temporal discretization is defined on time nodes $\{t^n\}_{n=0}^N$ with $t^0=0$ and $t^N=T$. Let $\dt^n = t^{n+1}-t^n$ for $n\in\calN\eqq\intset{0}{N-1}$, and denote $\calN_k\eqq \intset{k-1}{N-1}$ for $k\in\{1,\ldots,5\}$. For simplicity, we assume a uniform time step $\dt$.

\subsection{Numerical schemes}
We now present three schemes corresponding to the above problem classes. For fourth-order parabolic problems without a potential term, we combine the variational inequality formulation in space with BDF temporal discretizations, leading to schemes that preserve bounds and mass. In addition, for first-order temporal discretization, the energy stability is also satisfied. For problems with a potential term, we incorporate the SAV technique to enhance the stability while maintaining bounds and mass. For second-order parabolic problems, we introduce a fourth-order regularization, allowing the same framework to be applied.

\paragraph*{Fourth-order parabolic case without potential term.} Here we consider the model problem \eqref{eq:ch-linear}. Similar to the linear elliptic case, under suitable regularity assumptions \cite{Butz_2012}, the model problem can be reformulated as the following variational inequality: find $(u,w)\in H^1((0,T]; H^{-1}(\Dom))\cap L^2((0,T];H^2(\Dom;[a,b]))\cap L^\infty((0,T];H^1(\Dom)) \times L^2((0,T];H^1(\Dom))$ such that for almost every $t\in(0,T]$,
\begin{equation}\label{eq_model_without_f_VI}
    \begin{cases}
        \langle\partial_t u, v\rangle_{H^{-1},H^1} + (M(u)\nabla w , \nabla v) = 0, \quad \forall v\in H^1(\Dom),\\
        (w,\xi-u) \leq (\nabla u, \nabla ( \xi-u) ), \quad \forall \xi\in H^1(\Dom;[a,b]).
    \end{cases}
\end{equation}
The spatial discretization is obtained by replacing $H^1(\Dom)$ and $H^1(\Dom;[a,b])$ by $V_h^p$ and $V_h^{p,+}$, respectively. For the temporal discretization, we consider the BDF-type discretization with $M$ treated explicitly. The fully discrete problem then reads as follows: for all $n\in\calN_k$, find $(u_h^{n+1},w_h^{n+1})\in V_h^{p,+}\times V_h^p$ such that
\begin{equation}\label{eq_scheme_without_f}
    \begin{cases}
        ( \frac{\alpha_k u_h^{n+1} - A_k(u_h^n)}{\dt} , v_h ) + ( M(\Bar{u}_h^{n+1}) \nabla w_h^{n+1} , \nabla v_h) = 0, \qquad \forall v_h\in V_h^p,\\
        (w_h^{n+1},\xi_h-u_h^{n+1}) \leq (\nabla u_h^{n+1}, \nabla (\xi_h-u_h^{n+1}) ), \quad \forall \xi_h \in V_h^{p,+},
    \end{cases}
\end{equation}
where $\alpha_k$ and operators $A_k$, $k\in\{1,2,3,4,5\}$ are defined through the standard BDF$k$ convention, given by:
\begin{equation}
\begin{aligned}
    &\alpha_1=1, \quad \alpha_2=\frac{3}{2},\quad \alpha_3=\frac{11}{6},\quad \alpha_4=\frac{25}{12},\quad \alpha_5=\frac{137}{60},\\
    &A_1(u_h^n)=u_h^n,\quad  A_2(u_h^n)=2u_h^n-\frac{1}{2}u_h^{n-1},\quad A_3(u_h^n)=3u_h^n-\frac{3}{2}u_h^{n-1}+\frac{1}{3}u_h^{n-2},\\
    &A_4(u_h^n)=4u_h^n-3u_h^{n-1}+\frac{4}{3}u_h^{n-2}-\frac{1}{4}u_h^{n-3},\quad A_5(u_h^n)=5u_h^n-5u_h^{n-1}+\frac{10}{3}u_h^{n-2}-\frac{5}{4}u_h^{n-3}+\frac{1}{5}u_h^{n-4}
\end{aligned}
\end{equation}
and $\Bar{u}_h^{n+1} = \sum_{i\in\DoFSet}\varphi_i \min(b,\max(a,B_k(U_i^n)))$ is a bound-preserving modification of the standard extrapolation, obtained by projecting the extrapolated values onto the admissible interval $[a,b]$ at each node. Operator $B_k$ with $k\in\{1,2,3,4,5\}$ is the standard extrapolation operator given by
\begin{equation}
\begin{aligned}
    &B_1(u_h^n)=u_h^n,\quad B_2(u_h^n)=2u_h^n-u_h^{n-1},\quad B_3(u_h^n)=3u_h^n-3u_h^{n-1}+u_h^{n-2},\\
    &B_4(u_h^n)=4u_h^n-6u_h^{n-1}+4u_h^{n-2}-u_h^{n-3},\quad B_5(u_h^n)=5u_h^n-10u_h^{n-1}+10u_h^{n-2}-5u_h^{n-3}+u_h^{n-4}.
\end{aligned}
\end{equation}
\begin{remark}[mass-conservation and BDF]
    We observe that, the operator pair $(\alpha_k,A_k)$ in the above definition is mass-conservative, in the sense that if $\int_\Dom u_h^n=\cdots=\int_\Dom u_h^{n-k+1}$, we have $\int_\Dom \alpha_k u_h^{n+1} = \int_\Dom A_k(u_h^n)$.
\end{remark}

\paragraph*{General fourth-order parabolic case.} 
When $f\neq 0$, introducing the auxiliary variable $r = \sqrt{E_1(u)}$ \cite{Shen_Xu_Yang_2019} to enhance the stability, the system with sufficiently smooth solution can be reformulated as
\begin{equation*}
\begin{cases}
    \partial_t u =\DIV (M(u)\nabla w),\\
    w = -\Delta u + \frac{r}{\sqrt{E_1(u)}}f(u),\\
    \partial_t r= (\frac{f(u)}{2\sqrt{E_1(u)}},\partial_t u).
\end{cases}
\end{equation*}
We discretize this system using BDF time stepping, explicit extrapolation for nonlinear terms, and the variational inequality constraint. The resulting fully discrete scheme reads: for $n\in\calN_k$, find $(u_h^{n+1},r^{n+1},w_h^{n+1})\in V_h^{p,+}\times \Real \times V_h^p$ such that
\begin{equation}\label{eq_scheme}
    \begin{cases}
        ( \frac{\alpha_k u_h^{n+1} - A_k(u_h^n)}{\dt} , v_h ) + ( M(\Bar{u}_h^{n+1}) \nabla w_h^{n+1} , \nabla v_h) = 0, \qquad \forall v_h\in V_h^p,\\
        (w_h^{n+1},\xi_h-u_h^{n+1}) \leq (\nabla u_h^{n+1}, \nabla (\xi_h-u_h^{n+1}) ) + (\frac{r^{n+1}}{\sqrt{E_1(\Bar{u}_h^{n+1})}}f(\Bar{u}_h^{n+1}),\xi_h-u_h^{n+1}), \quad \forall \xi_h \in V_h^{p,+},\\
        \alpha_k r^{n+1} - A_k(r^n) = (\frac{f(\Bar{u}_h^{n+1})}{2\sqrt{E_1(\Bar{u}_h^{n+1})}}  , \alpha_k u_h^{n+1} - A_k(u_h^n) ),
    \end{cases}
\end{equation}
where $\alpha_k$, $A_k$ and $\Bar{u}_h^{n+1}$ are defined as in \eqref{eq_scheme_without_f}.

\paragraph*{Second-order parabolic case.} For the model problem \eqref{eq:second-order-parabolic}, we perturb it by adding a small fourth-order operator $\epsilon \Delta^2 u$ with $\epsilon>0$ in the PDE. Formally, this modification introduces a modeling error of order $O(\epsilon)$. We illustrate this idea by considering the Poisson problem $-\Delta u =f_1$ with periodic boundary condition in the torus $\mathbb{T}^d$. Through Fourier transform, the original PDE is equivalent to $\sum_{k\in\mathbb{Z}^d}|k|^2\hat{u}(k)e^{ik\xcoord} = \sum_{k\in\mathbb{Z}^d}\hat{f_1}(k) e^{ik\xcoord}$, with Fourier modes $\{\hat{u}(k)\}_{k\in\mathbb{Z}^d}$ and $\{\hat{f_1}(k)\}_{k\in\mathbb{Z}^d}$. Owing to the orthogonality of Fourier basis, for each mode $k\in \mathbb{Z}^d \setminus \{0\}$, we have $|k|^2\hat{u}(k)=\hat{f_1}(k)$. For the perturbed PDE $-\Delta u^\epsilon + \epsilon \Delta^2 u^\epsilon =f_1$, the same argument leads to $(|k|^2+\epsilon |k|^4)\hat{u}^\epsilon(k)=\hat{f_1}(k)$. Hence for each mode $k\neq 0$, the perturbation introduces an error of $O(\epsilon)$. This argument formally indicates that the perturbation introduces an $O(\epsilon)$ consistency error at the continuous level.

At the discrete level, different placements of $\epsilon$ lead to systems with varying conditioning. We consider a symmetric scaling to minimize the conditioning and motivate the following regularized model:
\begin{equation*}
    \begin{cases}
        \partial_t u = \DIV (K(u) \nabla u) + \sqrt{\epsilon} \Delta w,\\
        w=-\sqrt{\epsilon}\Delta u,
    \end{cases}
\end{equation*}
with the initial data $u_0$ and the homogeneous Neumann boundary conditions for $u$ and $w$. It follows that an artificial boundary condition for $\Delta u$ is imposed since the second-order parabolic problem does not have homogeneous Neumann boundary condition for $w\eqq -\sqrt{\epsilon}\Delta u$. However, based on our numerical observation on \eqref{eq_second_accuracy}, the model error introduced by such perturbation appears negligible, and it does not influence the convergence rate for smooth solutions.
Then, we reformulate the above model problem as the fourth-order cases: find $(u,w)\in H^1((0,T]; H^{-1}(\Dom))\cap L^2((0,T];H^2(\Dom;[a,b]))\cap L^\infty((0,T];H^1(\Dom)) \times L^2((0,T];H^1(\Dom))$ such that for almost every $t\in(0,T]$,
\begin{equation}\label{eq_model_second_order_VI}
    \begin{cases}
        \langle \partial_t u, v\rangle_{H^{-1},H^1} + \sqrt{\epsilon} (\nabla w , \nabla v) + (K(u)\nabla u , \nabla v) = 0, \quad \forall v\in H^1(\Dom),\\
        (w,\xi-u) \leq \sqrt{\epsilon}(\nabla u, \nabla ( \xi-u) ), \quad \forall \xi\in H^1(\Dom;[a,b]).
    \end{cases}
\end{equation}

Then, we discretize it as in \eqref{eq_scheme_without_f}, with $\epsilon=h^{p+1}$ to balance the perturbation error and conditioning, leading to the fully discretized problem: for all $n\in\calN_k$, find $(u_h^{n+1},w_h^{n+1})\in V_h^{p,+}\times V_h^p$ such that
\begin{equation}\label{eq_scheme_second_order}
    \begin{cases}
        ( \frac{\alpha_k u_h^{n+1} - A_k(u_h^n)}{\dt} , v_h ) + ( K(\Bar{u}_h^{n+1}) \nabla u_h^{n+1} , \nabla v_h) + \sqrt{h^{p+1}} (\nabla w_h^{n+1},\nabla v_h)   = 0, \qquad \forall v_h\in V_h^p,\\
        ( w_h^{n+1}, \xi_h-u_h^{n+1}) \leq \sqrt{h^{p+1}} (\nabla u_h^{n+1}, \nabla (\xi_h-u_h^{n+1}) ), \quad \forall \xi_h \in V_h^{p,+},
    \end{cases}
\end{equation}
with $\alpha_k$, $A_k$ the BDF operators, and $\Bar{u}_h^{n+1}$ the extrapolation approximation of $u_h^{n+1}$ constructed as in \eqref{eq_scheme_without_f}. 

Unlike the fourth-order case, the resulting scheme cannot be recast as a convex minimization problem. Therefore, its well-posedness is not covered by the present analytical framework. Nevertheless, the scheme exhibits stable behavior in all numerical experiments, which provides strong evidence of well-posedness.

Two components remain to complete the proposed schemes. The first concerns the construction of high-order, bound-preserving and mass-conservative initial data $(u_h^0,...,u_h^{k-1})$; the other concerns the construction of initial SAVs. We answer them one by one in the rest of this part.

\paragraph*{Initial data.} For BDF-type schemes, the initial data $u_h^0$ and $(u_h^1,...,u_h^{k-1})$ are in general prepared by the $L^2$-projection and the Runge-Kutta method, respectively. We refer to \cite{Thomee_2006} and \cite{Huang_Shen_2022} for details. The drawback of this initialization is the lack of the bound-preservation. Hence we propose two different strategies here. 

In the case $k=1,2$ and $u_0\in H^2(\Dom)$, we can generate $u_h^0$ by our stationary scheme \eqref{eq_scheme_stationary}, \ie we use the stationary scheme \eqref{eq_scheme_stationary} as a bound-preserving and mass-conservative alternative to the $L^2$-projection. More precisely, we define $u_h^0\in V_h^{p,+}$ by solving
\begin{equation*}
    \begin{cases}
        (u_h^0,v_h) + (\nabla w_h,\nabla v_h)= ( u_0 + \Delta^2 u_0,v_h), \qquad \forall v_h \in V_h^p,\\
        (w_h,\xi_h-u_h^0) \leq (\nabla u_h^0, \nabla (\xi_h-u_h^0)), \qquad \forall \xi_h\in V_h^{p,+}.
    \end{cases}
\end{equation*}
By the construction, $u_h^0$ is bound-preserving and mass-conservative.  Then $u_h^1$ can be prepared by \eqref{eq_scheme_without_f}, \eqref{eq_scheme} or \eqref{eq_scheme_second_order} with $k=1$.

For other cases, we consider a bound-preserving and mass-conservative postprocessing to prepare initial data. We first compute $\Tilde{u}_h^0$ and $(\Tilde{u}_h^1,...,\Tilde{u}_h^{k-1})$ by the $L^2$-projection and the Runge-Kutta method, respectively. Then, we apply the optimization-based postprocessing
\begin{equation*}
    u_h^j = \argmin_{v_h\in V_h^{p,+}, \int_\Dom v_h=\int_\Dom \Tilde{u}_h^j} \norm{v_h-\Tilde{u}_h^j}^2, \qquad j\in \{0,...,k-1\}
\end{equation*}
to define the initial data. The above postprocessing is motivated by \cite{Liu_Riviere_Shen_Zhang_2024}, and it is numerically observed to achieve the optimal convergence rate.

\paragraph*{Initialization of SAV.} Since at the continuous level, $r=\sqrt{E_1(u)}$, we set $r^n=\sqrt{E_1(u_h^n)}$ for $n\in \{0,...,k-1\}$. This choice maintains the optimal approximation properties of $r^n$, once $u_h^n$ has such properties.

\subsection{Fundamental properties}\label{Sec:fundamental_properties}

We first establish the well-posedness and structure-preserving properties of the scheme \eqref{eq_scheme} for fourth-order parabolic problems with a potential term, as this formulation is the most general. The scheme \eqref{eq_scheme_without_f} can be treated as a special case of \eqref{eq_scheme}.

For the regularized second-order scheme \eqref{eq_scheme_second_order}, the presence of the additional diffusion term prevents a direct reformulation as a convex minimization problem. Therefore, its well-posedness is not covered by the present analytical framework. Instead, we prove the bound-preservation and mass-conservation by assuming the existence of the numerical solution.

\paragraph*{General fourth-order parabolic case.}

The fundamental properties of the proposed scheme \eqref{eq_scheme} can be derived following the same argument used for the stationary scheme \eqref{eq_scheme_stationary}.

The following lemma establishes the well-posedness of the variational inequality \eqref{eq_scheme}.
\begin{lemma}[Properties of \eqref{eq_scheme}]\label{lemma:well_posedness}
    For all $n\in\calN_k$, the variational inequality \eqref{eq_scheme} admits a solution $(u_h^{n+1},r^{n+1},w_h^{n+1})$. Moreover, $u_h^{n+1}$ is unique, bound-preserving, mass-conservative and the unique minimizer of the following problem:
    \begin{equation}\label{eq_minimization_dis}
        \min_{(v,s) \in X^{n+1}  } \Big\{  \frac{\alpha_k}{2}\norm{\nabla v}^2 + \alpha_k s^2 + \frac{1}{2\dt} \norm{\alpha_k v - A_k(u_h^n)}_{H^{-1,n+1}_h}^2 \Big\},
    \end{equation}
    where
    \begin{equation}
        X^{n+1}=\{(v,s)\in V_h^{p,+}\times \Real \ | \  \int_\Dom v = \int_\Dom u_h^n \ , \  \alpha_k s - A_k(r^n) = (\frac{f(\Bar{u}_h^{n+1})}{2\sqrt{E_1(\Bar{u}_h^{n+1})}}  , \alpha_k v - A_k(u_h^n) )  \}.
    \end{equation}
    
    In addition, for $k=1$, \ie the first-order temporal discretization, we have the modified energy stability
        \begin{equation}\label{eq_modified_energy_stability}
            \frac{1}{2}\norm{\nabla u_h^{n+1}}^2 +(r^{n+1})^2 \leq \frac{1}{2}\norm{\nabla u_h^n}^2 +(r^n)^2.
        \end{equation}
\end{lemma}

\begin{remark}[SAV stabilization]
    We note that the modified energy stability \eqref{eq_modified_energy_stability} holds true unconditionally, similar to its application on parabolic PDEs in \cite{Shen_Xu_Yang_2019,Huang_Shen_2021}. Although the modified energy stability is not proved for high-order temporal discretizations, they are numerically observed to have an improved energy stability, even for a challenging Cahn--Hilliard test in \eqref{eq_CH_log_BP}.
\end{remark}

\begin{proof}
    The proof follows the same variational framework as in Lemma~\ref{lemma:well_posedness_stationary}. For completeness, we outline the main steps, emphasizing the additional coupling introduced by the auxiliary variable $r$.

    \paragraph*{Step 1: well-posedness of the minimization problem.}
    We first verify that the admissible set $X^{n+1}$ is nonempty, convex, and closed.
    
    \textbf{Nonemptiness.} 
    We notice that $(u_h^n,\frac{1}{\alpha_k} ( (\frac{f(\Bar{u}_h^{n+1})}{2\sqrt{E_1(\Bar{u}_h^{n+1})}}  , \alpha_k u_h^n - A_k(u_h^n) )) + A_k(r^n) )$ satisfies all constraints of $X^{n+1}$, hence $X^{n+1}$ is nonempty.
    
    \textbf{Convexity.}  
    Let $(v_1,s_1),(v_2,s_2)\in X^{n+1}$ and $\theta\in[0,1]$. Then $\theta v_1+(1-\theta)v_2\in V_h^{p,+}$ and satisfies the mass constraint by linearity. Moreover, the linear constraint on $s$ is preserved by convex combinations. Hence $X^{n+1}$ is convex.
    
    \textbf{Closedness.}  
    Since $V_h^{p,+}$ is closed in $V_h^p$ and the additional constraints are linear, $X^{n+1}$ is closed.
    
    Therefore, the minimization problem \eqref{eq_minimization_dis} consists of minimizing a strictly convex functional over a finite-dimensional, nonempty, closed, and convex set. It thus admits a unique minimizer $(u_h^{n+1},r^{n+1})$.
    
    \paragraph*{Step 2: equivalence with the variational inequality.}
    We show that the minimization problem \eqref{eq_minimization_dis} is equivalent to the variational formulation \eqref{eq_scheme}. For notational simplicity, we write
    \[
    M = M(\Bar{u}_h^{n+1}), \qquad f = f(\Bar{u}_h^{n+1}), \qquad E_1 = E_1(\Bar{u}_h^{n+1}).
    \]
    
    The minimization problem can be rewritten in the equivalent form
    \begin{align*}
    \min_{(u_h,v_h,s)\in V_h^p \times V_h^p \times \mathbb{R}} \;& \frac{\alpha_k}{2}\|\nabla u_h\|^2 + \alpha_k s^2 + \frac{\dt}{2}(M\nabla v_h,\nabla v_h) \\
    \text{s.t. } 
    & (\alpha_k u_h - A_k(u_h^n),\chi_h) + \dt (M\nabla v_h,\nabla \chi_h) = 0, \quad \forall \chi_h\in V_h^p, \quad \int_\Dom v_h = 0,\\
    & a \le U_i \le b, \quad \forall i\in \DoFSet,\quad u_h=\sum_{i\in\DoFSet} U_i\varphi_i,\\
    & \alpha_k s - A_k(r^n) = \Big(\frac{f}{2\sqrt{E_1}},\alpha_k u_h - A_k(u_h^n)\Big).
    \end{align*}
    
    By standard convex optimization theory with box and linear constraints, the minimizer satisfies the associated KKT system. Eliminating the Lagrange multipliers associated with the box constraints yields the variational inequality: find $(u_h,w_h)\in V_h^{p,+}\times V_h^p$ such that
    \begin{align*}
    &(\tfrac{\alpha_k u_h - A_k(u_h^n)}{\dt}, \chi_h) + (M\nabla w_h,\nabla \chi_h) = 0, \quad \forall \chi_h\in V_h^p,\\
    &(w_h,\xi_h-u_h) \leq (\nabla u_h,\nabla(\xi_h-u_h)) + \Big(\frac{s}{\sqrt{E_1}}f,\xi_h-u_h\Big), \quad \forall \xi_h\in V_h^{p,+},
    \end{align*}
    together with the scalar constraint
    \[
    \alpha_k s - A_k(r^n) = \Big(\frac{f}{2\sqrt{E_1}},\alpha_k u_h - A_k(u_h^n)\Big).
    \]
    
    This is precisely the scheme \eqref{eq_scheme}. Therefore, the minimization problem and the variational formulation are equivalent, and the well-posedness of \eqref{eq_scheme} follows.

    \paragraph*{Step 3: structure-preservation.}
    For the structure-preserving properties, we notice that $(u_h^{n+1},r^{n+1})$ is an element in $X^{n+1}$, hence $u_h^{n+1}$ is mass-conservative and bound-preserving. Then, for $k=1$, we notice that $(u_h^{n+1},r^{n+1})$ is a minimizer of the problem \eqref{eq_minimization_dis}, and $(u_h^n,r^n)$ is also in $X^{n+1}$, we have
    \begin{align*}
        \frac{1}{2}\norm{\nabla u_h^{n+1}}^2 &+ (r^{n+1})^2 + \frac{1}{2\dt}\norm{u_h^{n+1}-u_h^n}_{H_h^{-1,n+1}}^2\\
        &= \min_{(v_h,s)\in X^{n+1}} \Big( \frac{1}{2}\norm{\nabla v_h}^2 + s^2 + \frac{1}{2\dt}\norm{v_h-u_h^n}_{H_h^{-1,n+1}}^2 \Big) \leq \frac{1}{2}\norm{\nabla u_h^n }^2 + (r^n )^2.
    \end{align*}
    It follows the expected modified energy stability \eqref{eq_modified_energy_stability}, which concludes the proof.
    
\end{proof}

\paragraph{Fourth-order parabolic case without potential term.}
The well-posedness and structure-preserving properties of \eqref{eq_scheme_without_f} can be proved by considering them as a simplified case of Lemma \ref{lemma:well_posedness}.
\begin{lemma}[Properties of \eqref{eq_scheme_without_f}]\label{lemma:properties_scheme_without_f}
    For all $n\in\calN_k$, the variational inequality \eqref{eq_scheme_without_f} has a solution $(u_h^{n+1},w_h^{n+1})$, and $u_h^{n+1}$ is unique, bound-preserving and mass-conservative. Moreover, it is the unique minimizer of the following problem:
    \begin{equation}\label{eq_minimization_dis_without_f}
        \min_{v \in Y^{n+1}  } \Big\{  \frac{\alpha_k}{2}\norm{\nabla v}^2 + \frac{1}{2\dt} \norm{\alpha_k v - A_k(u_h^n)}_{H^{-1,n+1}_h}^2 \Big\},
    \end{equation}
    where
    \begin{equation}
        Y^{n+1}=\{v\in V_h^{p,+} \ | \  \int_\Dom v = \int_\Dom u_h^n\}.
    \end{equation}
    In addition, for $k=1$, \ie the first-order temporal discretization, we have the energy stability
    \begin{equation}
        \frac{1}{2}\norm{\nabla u_h^{n+1}}^2 \leq \frac{1}{2}\norm{\nabla u_h^n}^2.
    \end{equation}
\end{lemma}

\begin{proof}
    This lemma is a special case of Lemmas \ref{lemma:well_posedness} with $f=0$ and $r=0$.
\end{proof}

\paragraph{Second-order parabolic case.}
We then discuss the properties of the proposed scheme \eqref{eq_scheme_second_order} for the second-order parabolic problem \eqref{eq:second-order-parabolic}. Unlike the fourth-order parabolic problems, the appearance of the additional diffusion term $(K(u)\nabla u,\nabla v)$ makes it difficult to reformulate the proposed scheme as a single variable minimization problem as in Lemmas \ref{lemma:well_posedness} and \ref{lemma:properties_scheme_without_f}. However, when we practically solve the variational inequality \eqref{eq_scheme_second_order} by the primal--dual active set method, the corresponding linear system has similar structure as the one in \eqref{eq_scheme_without_f}. The scheme \eqref{eq_scheme_second_order} is therefore expected to be well-posed. Nevertheless, assuming the well-posedness of the proposed scheme, we have the bound-preservation and mass-conservation properties:
\begin{lemma}[Properties of \eqref{eq_scheme_second_order}]\label{lemma:properties_scheme_second_order}
    Assume that \eqref{eq_scheme_second_order} admits a solution. Then, for all $n\in\calN_k$, the solution  $u_h^{n+1}$ in \eqref{eq_scheme_second_order} is bound-preserving and mass-conservative. 
\end{lemma}

\begin{proof}
    The argument is identical to that of the fourth-order case once existence is assumed.
    Indeed, the bound-preservation follows from the fact that $u_h^{n+1}\in V_h^{p,+}$. By taking $v_h=1\in V_h^p$ in the first equation of \eqref{eq_scheme_second_order}, we derive the mass-conservation property.
\end{proof}

\section{Numerical experiments}\label{sec:numerics}
In this section, we assess the accuracy, structure-preserving properties, and robustness of the proposed schemes, and compare them with representative methods from the literature. For the fourth-order schemes \eqref{eq_scheme_stationary}, \eqref{eq_scheme_without_f}, and \eqref{eq_scheme}, the experiments are designed to validate the theoretical results established in the previous sections, including convergence behavior, bound-preservation, mass-conservation, and energy stability. For the regularized second-order scheme \eqref{eq_scheme_second_order}, whose analysis is more limited, the experiments are intended primarily to evaluate the effectiveness of the regularization strategy and its practical ability to preserve the desired structural properties.

Although the numerical schemes are established for model problems with mild assumptions like uniform ellipticity, and time-independent bounds $[a,b]$, we also consider tests which do not satisfy such mild assumptions, to illustrate the performance of the proposed schemes on challenging benchmarks.

All computations are carried out in the \texttt{Julia} programming language using the \texttt{Gridap.jl} library \cite{Verdugo_Badia_2022}. 

Unless otherwise specified, errors are measured in relative form; for example, the relative $L^2$-error is defined by $\norm{u-u_h}/\norm{u}$. Homogeneous Neumann boundary conditions are imposed in all tests unless stated otherwise. For tests without a priori information on bounds $[a,b]$, these bounds are determined by the numerical reference solutions with sufficiently small $h$ and $\tau$. The bound-preservation and mass-conservation are satisfied up to rounding error, hence we systematically decide to not display them. The iteration number needed for primal--dual active set method is uniformly bounded by $5$ in all tests, and are systematically omitted for brevity.

The section is organized into four parts, corresponding to the schemes \eqref{eq_scheme_stationary}, \eqref{eq_scheme_without_f}, \eqref{eq_scheme}, and \eqref{eq_scheme_second_order}, respectively.

\subsection{Stationary case}
In this section, we assess the performance of the proposed scheme \eqref{eq_scheme_stationary}. First, we test the accuracy of the proposed scheme on low-quality meshes. Then, problems with discontinuous data are used to illustrate the reliability of our scheme for non-smooth problems.

\paragraph*{2D smooth test on meshes with obtuse triangles.} 
To assess the accuracy of the proposed method on low-quality meshes, we consider
\begin{equation}\label{eq_test_stationary_smooth_2d}
    u=\DIV ((1+x)\nabla w)+f_1,\qquad w=-\Delta u,
\end{equation}
with analytic solution 
\begin{equation}
    u(x,y)=\cos(4\pi x) \cos(4\pi y),
\end{equation}
in the domain $(0,1)^2$.  $f_1$ is computed based on $u$ and $w$.

Here, we consider non-uniform meshes to introduce mesh irregularity and test robustness with respect to mesh quality. Each mesh is constructed as follows: We start with a uniform mesh $\calT_h$ composed of isosceles right triangles, then we perturb each interior node of the domain by a random vector in the ball $B(0,0.2h)$. 

In Figure \ref{fig:accuracy_stationary_obtuse}, we illustrate an example of $20$ nodes in each direction in the left panel, while the middle and right panels display the convergence curves with $p\in\{1,2,3,4\}$ in the $L^2$- and $H^1$-norms in the middle and right panel.

\begin{figure}[htb]
    \centering
    \begin{subfigure}{0.3\textwidth}
        \includegraphics[width=\textwidth]{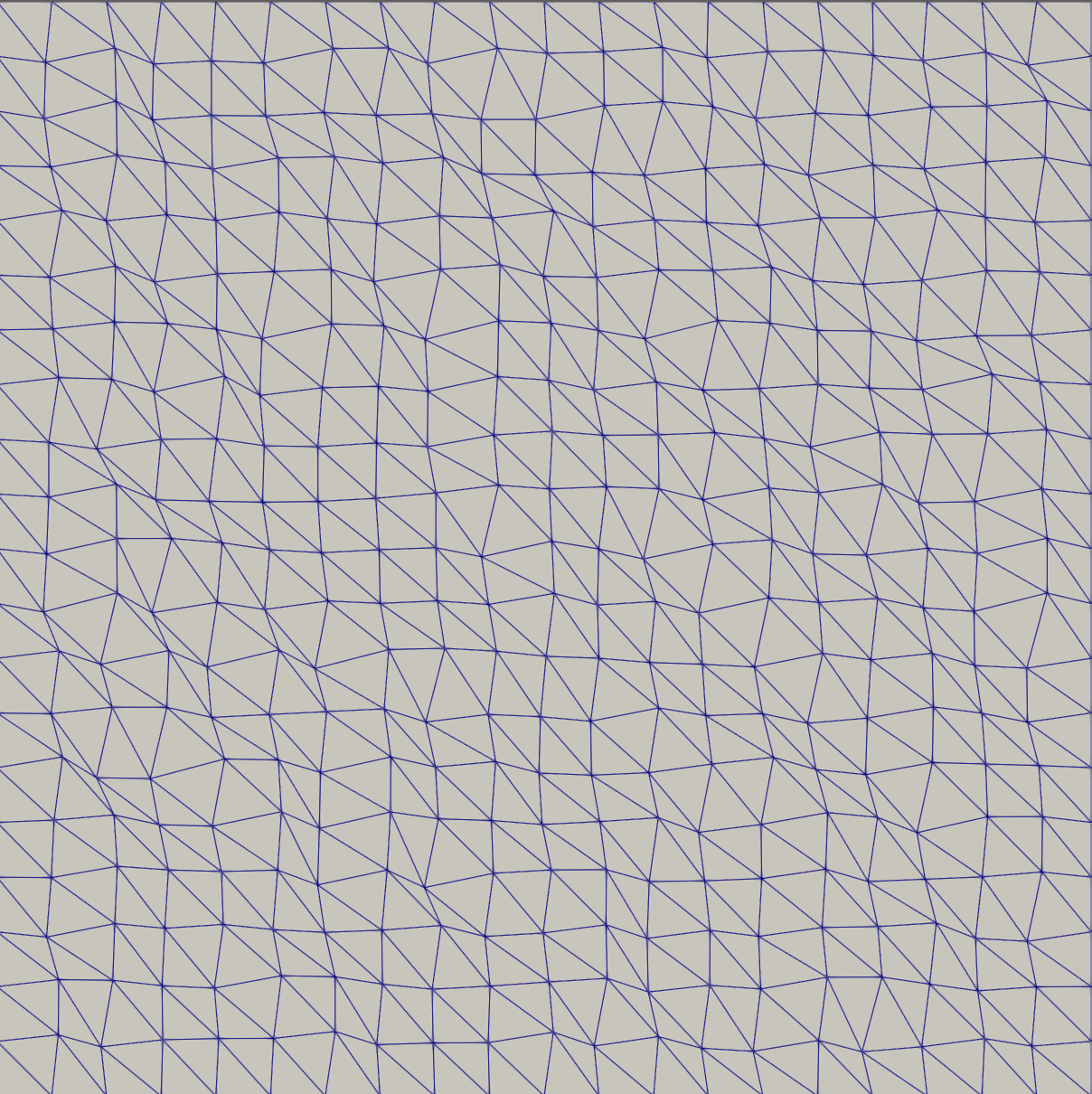}
        \caption{Mesh with obtuse triangles.}
    \end{subfigure}
    \begin{subfigure}{0.3\textwidth}
        \includegraphics[width=\textwidth]{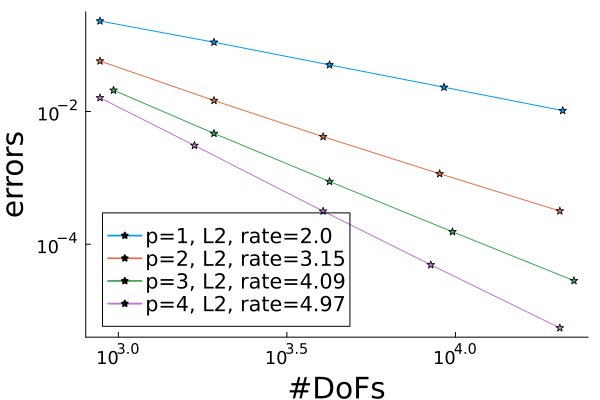}
        \caption{Convergence test, $L^2$-norm.}
    \end{subfigure}
    \begin{subfigure}{0.3\textwidth}
        \includegraphics[width=\textwidth]{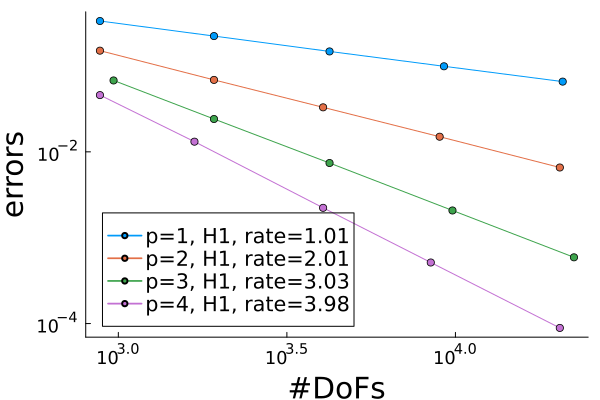}
        \caption{Convergence test, $H^1$-norm.}
    \end{subfigure}
    \caption{Accuracy test on smooth solution of \eqref{eq_test_stationary_smooth_2d} on non-uniform meshes.}
    \label{fig:accuracy_stationary_obtuse}
\end{figure}
From Figure \ref{fig:accuracy_stationary_obtuse}, we observe the expected convergence rates in both $L^2$- and $H^1$-norms, for bound-preserving and mass-conservative numerical solutions. In addition, several cells in the mesh have large angles, yet the observed convergence rates are unaffected by this poor mesh quality.

\paragraph*{1D discontinuous data test.} Here we assess the stability and accuracy of the proposed scheme with discontinuous data. We set $\Dom=(-2,2)$,
\begin{equation}\label{eq_test_stationary_discontinuous_1d}
    u=\DIV(M\nabla w) +f_1, \qquad w=-\Delta u + f_2,
\end{equation}
with
\begin{equation}
    M(x)=\begin{cases}
        2,\quad \text{if } |x|<0.5,\\
        \epsilon,\quad \text{otherwise},
    \end{cases}\quad 
     f_1(x)=\begin{cases}
        1,\quad \text{if } 0<x<1,\\
        0,\quad \text{otherwise},
    \end{cases}\quad 
    f_2(x)=\begin{cases}
        5,\quad \text{if } -1<x<0,\\
        0,\quad \text{otherwise},
    \end{cases}
\end{equation}
and $\epsilon>0$. 

In Figure \ref{fig:stationary_discontinuous}, we present the reference solutions produced by $\mathbb{P}_3$-FEM on $5000$ cells with various $\epsilon$ in the left panel; the convergence test is illustrated in the right panel, by comparing the numerical solutions with the reference solutions, for $p=1$, $\epsilon\in\{10^{-1},10^{-3},10^{-5},10^{-7}\}$, on uniform meshes. 

\begin{figure}
    \centering
    \begin{subfigure}{0.4\textwidth}
        \includegraphics[width=\textwidth]{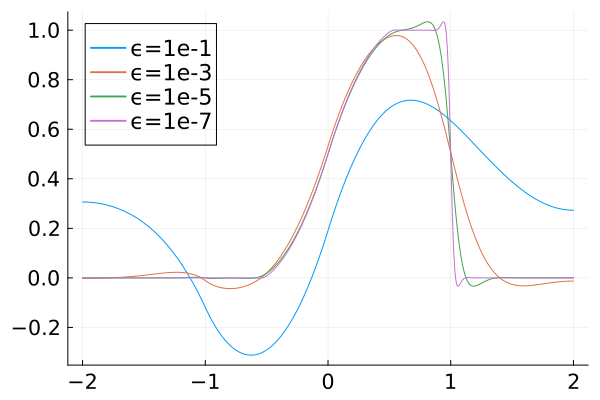}
        \caption{Reference solutions.}
    \end{subfigure}
    \begin{subfigure}{0.4\textwidth}
        \includegraphics[width=\textwidth]{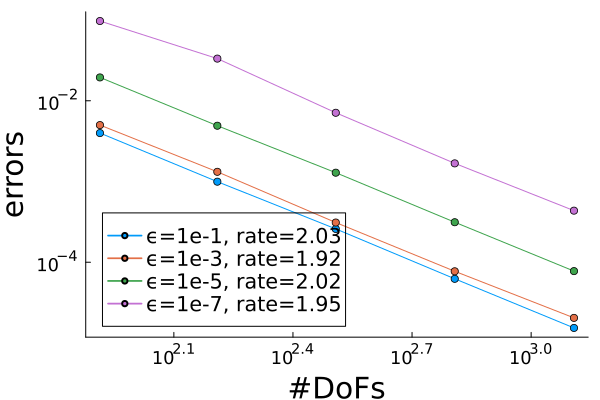}
        \caption{Convergence rate in the $L^2$-norm, $p=1$.}
    \end{subfigure}
    \caption{Accuracy tests on \eqref{eq_test_stationary_discontinuous_1d} with discontinuous data.}
    \label{fig:stationary_discontinuous}
\end{figure}

From Figure \ref{fig:stationary_discontinuous}, we observe the expected convergence behavior in all settings, which demonstrates the stability and accuracy of the proposed scheme on the problems with discontinuous data and local small mobility. Moreover, under the same discretization setting, smaller values of $\epsilon$ lead to larger errors, which is consistent with the fact that strongly contrasting coefficients require finer meshes to resolve.

\subsection{Fourth-order parabolic problem without potential term: lubrication-type equation}
In this section, we test the proposed scheme \eqref{eq_scheme_without_f} on the well-known lubrication-type equations, \ie $M(v)=v^\beta$ with $\beta>0$ and $f(v)=0$. As stated in \cite{Cheng_Shen_2022_1}, this is a challenging test for $\beta=\frac{1}{2}$, owing to the following reasons: 1. the given problem may have finite temporal singularity, which in general requires a regularization of $M$; 2. the negative value of $u$ is not allowed, since in the mobility $M$, we cannot take the square root of a negative value in $\Real$; 3. Standard bound-preserving techniques such as cut-off do not guarantee an energy stable solution; 4. a relaxation on the lower bound $0$ is needed in practice for bound-preserving schemes, \eg for $\dt=10^{-4}$, the lower bound cannot be smaller than $10^{-2}$ in \cite{Cheng_Shen_2022_1}.

We assess the accuracy of the proposed scheme on a manufactured solution, in the first part. In the second part, we evaluate the performance of proposed method on a challenging singular test, without any additional modification of the equation or any substantial relaxation of the lower bound.

\paragraph*{Accuracy test on smooth solution.} Here we assess the accuracy of the proposed scheme by considering the equation
\begin{equation}\label{eq_lubrication_accuracy}
    \partial_t u = \DIV(u\nabla w) + f_1,\quad w=-\Delta u,
\end{equation}
with the exact solution
\begin{equation}
    u(x,y,t)=(1+\cos(x)\cos(y))\cos(t),
\end{equation}
in the domain $(0,2\pi)^2$ with final time $T=1$. The source term $f_1$ is constructed from $u$ and $w$. The bounds are given by $a(t)=0$ and $b(t)=2\cos(t)$, and are therefore time-dependent.

In Figure \ref{fig:accuracy_Lubrication}, we present the convergence test on structured meshes with $\dt=\frac{h}{2}$, $p\in\{1,2,3,4\}$, and $k=p+1$, that is, using a BDF$p+1$ discretization in time.

\begin{figure}[htb]
    \centering
    \includegraphics[width=0.5\linewidth]{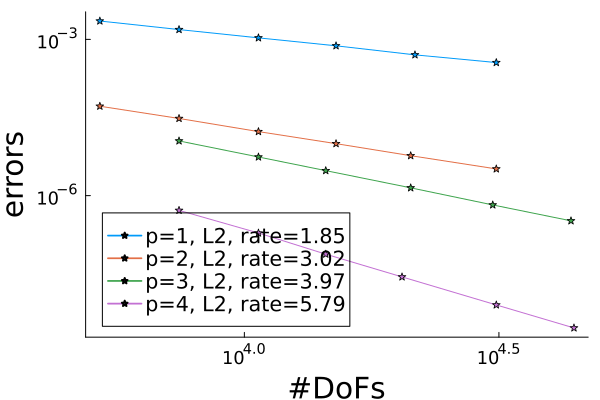}
    \caption{Accuracy test for the lubrication-type equation \eqref{eq_lubrication_accuracy}.}
    \label{fig:accuracy_Lubrication}
\end{figure}

We observe from Figure \ref{fig:accuracy_Lubrication} that the expected convergence rate in the $L^2$-norm at the final time for $p\in\{1,2,3\}$ are obtained, and there is even a slight superconvergence for $p=4$. Therefore, the numerical results confirm the expected convergence behavior of the proposed scheme.

\paragraph{Singular test with $\beta=0.5$.} Here we test the stability, efficiency and the structure-preservation of the proposed scheme, by considering
\begin{equation}\label{eq_lubrication_singular}
    \partial_t u = \DIV(\sqrt{u}\nabla w),\quad w=-\Delta u,
\end{equation}
with the initial condition 
\begin{equation}
    u_0(x)= 0.8-\cos(\pi x) + 0.25 \cos(2\pi x),
\end{equation}
in the domain $\Dom=(-1,1)$, and the homogeneous Neumann boundary condition, as in \cite[eq. (4.28)]{Cheng_Shen_2022_1}. We note that the exact solution exhibits a singular behavior at $t\approx 7.4\times 10^{-4}$.

In Figure \ref{fig:Lubrication_1D}, we illustrate the numerical result for this model problem. The simulations are made with $50$ spatial nodes, $p=1$, $k=1$, and various $\dt\in\{10^{-2},10^{-3},10^{-4},10^{-5}\}$, with lower bound $0$ relaxed by the tolerance $10^{-15}$. In the left panel, we present the computed numerical solutions at different times with $\dt=10^{-4}$; in the middle panel, the energy evolution is illustrated for $\dt\in \{ 10^{-2},10^{-3},10^{-4},10^{-5} \}$; in the right panel, the iteration needed in each time step is presented for $\dt\in \{ 10^{-3},10^{-4}, 10^{-5} \}$, for $t\in[0,10^{-2}]$. In the right panel, the evolution of iteration number for $t>10^{-2}$ and for the case $\dt=10^{-2}$ are omitted, since they are identically zero.

\begin{figure}[htb]
    \begin{subfigure}{0.3\textwidth}
        \includegraphics[width=\textwidth]{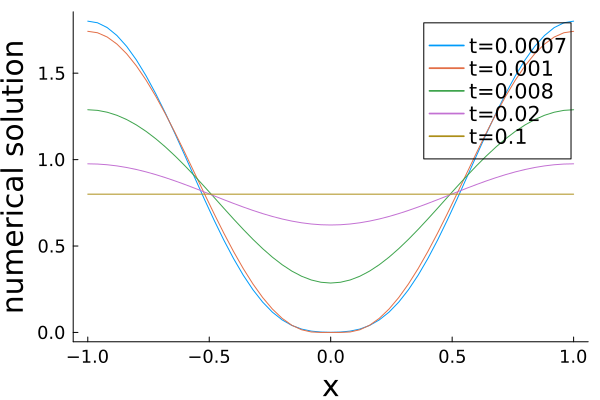}
        \caption{Numerical solution, $\dt=10^{-4}$.}
    \end{subfigure}
    \begin{subfigure}{0.3\textwidth}
        \includegraphics[width=\textwidth]{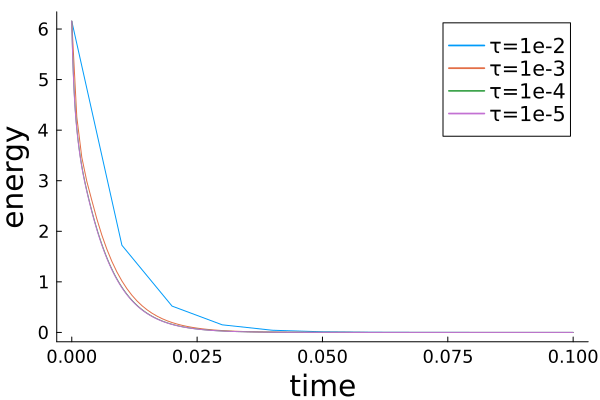}
        \caption{Energy evolution for different $\dt$.}
    \end{subfigure}
    \begin{subfigure}{0.3\textwidth}
        \includegraphics[width=\textwidth]{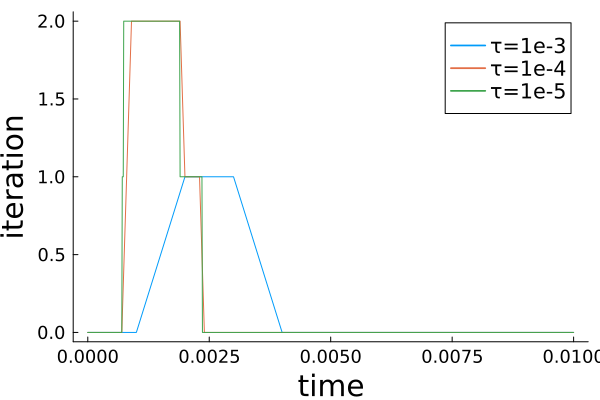}
        \caption{Iteration number for different $\dt$.}
    \end{subfigure}
    \caption{Simulation for temporal singular lubrication-type equation \eqref{eq_lubrication_singular}.}
    \label{fig:Lubrication_1D}
\end{figure}

We observe from Figure \ref{fig:Lubrication_1D} that: 1. the numerical solutions are in good agreement with the reference solutions in \cite{Cheng_Shen_2022_1,Zhornitskaya_Bertozzi_2000}, and the temporal singularity at $t\approx 7.4\times 10^{-4}$ does not cause difficulties for any tested value of $\dt$, even we did not use any regularization on the mobility $M$ or restrictive relaxation on the lower bound; 2. the energy dissipation is observed for entire range of $\dt$, even for $\dt=10^{-2}$, which illustrates the stability of the proposed scheme; 3. the iterative solver for the variational inequality is activated only near the temporal singularity, and the iteration number remains uniformly small (at most 2 in our tests).
These observations demonstrate the stability, efficiency and accuracy of the proposed scheme.

\subsection{General fourth-order parabolic problem: Cahn--Hilliard equation}

Here we consider the Cahn--Hilliard equation and its variants, to illustrate the numerical performance of our scheme \eqref{eq_scheme}. We consider two numerical tests in this section. The first one is an accuracy test on a manufactured smooth solution with nonlinear mobility and polynomial potential, and the second one is a challenging simulation considered in \cite{Cheng_Shen_2022}, with logarithmic potential. 

\paragraph*{Accuracy test on smooth solution.} Here we test the accuracy of the proposed method in 
\begin{equation}\label{eq_CH_accuracy}
    \partial_t u = \DIV ((1-u^2)\nabla w) + f_1, \quad w= -\Delta u + \frac{1}{4}(u^2-1)^2,
\end{equation}
known as the Cahn--Hilliard equation with nonlinear mobility and polynomial potential. We consider the exact solution
\begin{equation}
    u(x,y,t)=\cos(x)\cos(y)\cos(t),
\end{equation}
and $f_1$ is determined by $u$. The domain is $(0,2\pi)^2$ and the final time is $T=1$. The bounds of the above problem are given by $a(t)=-\cos(t),b(t)=\cos(t)$. 

In Figure \ref{fig:accuracy_CH}, we illustrate the error in the $L^2$-norm at the final time, with $C_0=1$, $\dt=\frac{h}{2}$,  $p\in\{1,2,3,4\}$ and $k=p+1$ (\ie BDF$p+1$ in time and $p$-th order FEM in space), on structured meshes. 

\begin{figure}[htb]
    \centering
    \includegraphics[width=0.5\linewidth]{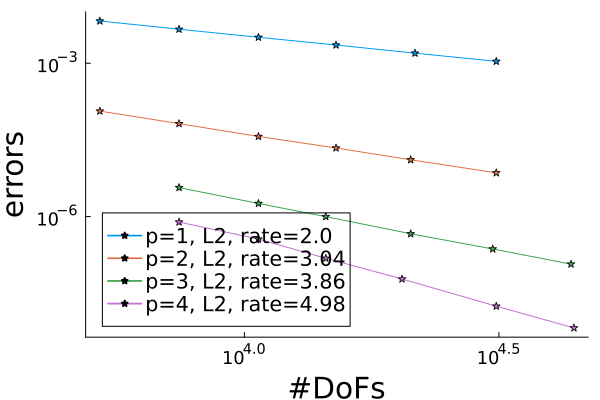}
    \caption{Accuracy test for the Cahn--Hilliard equation \eqref{eq_CH_accuracy}.}
    \label{fig:accuracy_CH}
\end{figure}

We observe from Figure \ref{fig:accuracy_CH} the expected convergence rate in the $L^2$-norm at the final time for $p\in\{1,2,3,4\}$. This confirms the accuracy of the proposed scheme. The convergence behavior of \eqref{eq_scheme} is similar to the potential-free case \eqref{eq_scheme_without_f}. This indicates that the introduction of the SAV technique does not degrade the accuracy of the BDF-type discretization, while providing additional modified energy stability.

\paragraph*{Logarithmic potential case.}
Here we assess the stability and the structure-preservation of the proposed method by testing
\begin{equation}\label{eq_CH_log_BP}
    \partial_t u = \DIV((1-u^2)\nabla w) , \quad w= -0.01\Delta u + \log(1+u) -\log(1-u)-5u,
\end{equation}
with the initial condition 
\begin{equation}
    u_0(x,y) = 0.2 +0.05\mathrm{rand}(-1,1),
\end{equation}
in the domain $(0,1)^2$ and the final time $T=1.5\times 10^{-1}$.
We note that the values of the solution should stay in $(-1,1)$, and the mobility tends to zero for $u$ taking values near $-1$ and $1$.

We set $C_0=1$ in the potential energy \eqref{eq:energy}, use $p=1$, $50$ nodes in each direction for spatial discretization, and set $k=2$ \ie BDF2 in the scheme \eqref{eq_scheme}. The upper and lower bounds are relaxed by $10^{-10}$. 

In Figure \ref{fig:CH_log_mobility_date}, we present several diagnostics for the proposed scheme.
Panel \ref{fig_CH_a} compares the energy evolution conducted with $\dt=10^{-4}$, for the original energy delivered by 1). our scheme \eqref{eq_scheme}, marked as `variational inequality + SAV' in the figure; 2). the variational inequality technique without SAV as \eqref{eq_scheme} with $r^n=0$ for $n\in\calN$, marked as `variational inequality' in the figure; 3). bound-preserving and mass-conservative Lagrange multiplier technique in \cite{Cheng_Shen_2022_1}, marked as `Lagrange multiplier' in the figure. 
Panel \ref{fig_CH_b} shows the evolution of the maximum and minimum values of the numerical solution, together with a reference solution computed with a smaller time step.
Panel \ref{fig_CH_c} illustrates the mass difference $\int_\Dom u_h^n-\int_\Dom u_h^0$, demonstrating mass-conservation, with $\dt=10^{-4}$; 
Panel \ref{fig_CH_d} reports the number of iterations required by the primal–dual active set solver at each time step for $\dt\in \{10^{-4}, 5\times 10^{-5}, 2\times 10^{-5}\}$, for the proposed scheme.




\begin{figure}[htb]
    \centering
    \begin{subfigure}{0.4\textwidth}
        \includegraphics[width=\textwidth]{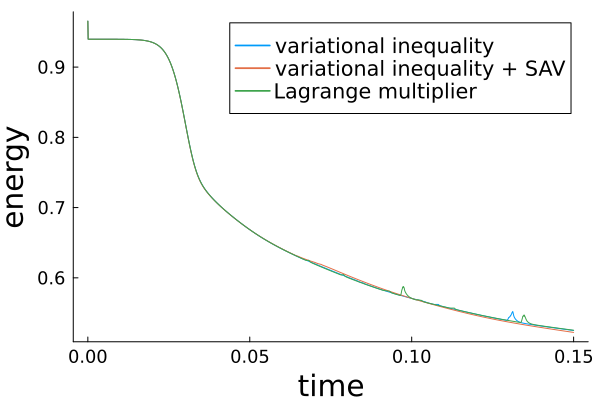}
        \caption{Comparison of energy evolution, $\dt=10^{-4}$.}
        \label{fig_CH_a}
    \end{subfigure}
    \begin{subfigure}{0.4\textwidth}
        \includegraphics[width=\textwidth]{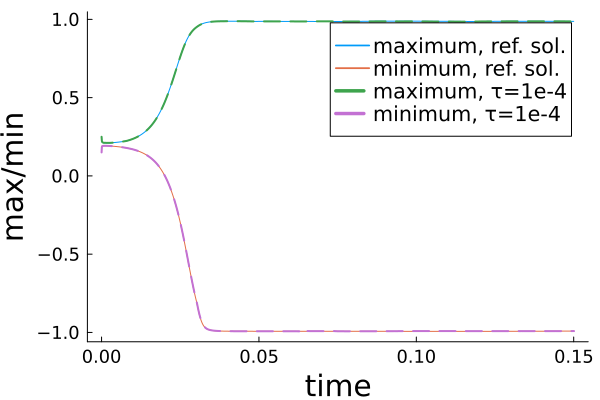}
        \caption{Maximum and minimum of $u_h^n$.}
        \label{fig_CH_b}
    \end{subfigure}
    
    \begin{subfigure}{0.4\textwidth}
        \includegraphics[width=\textwidth]{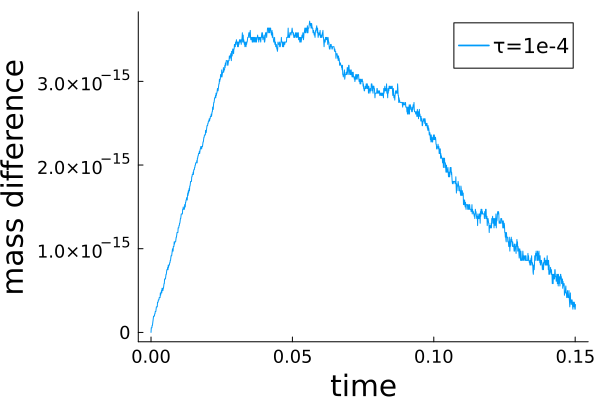}
        \caption{Mass-conservation of $u_h^n$.}
        \label{fig_CH_c}
    \end{subfigure}
    \begin{subfigure}{0.4\textwidth}
        \includegraphics[width=\textwidth]{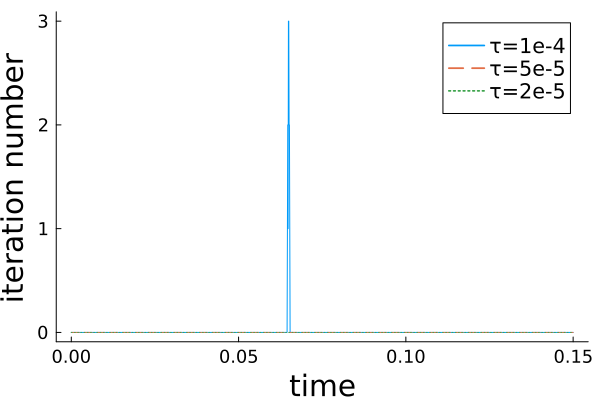}
        \caption{Comparison of iteration number.}
        \label{fig_CH_d}
    \end{subfigure}
    
    \caption{Computational data of Cahn--Hilliard equation with non-constant mobility and logarithmic potential \eqref{eq_CH_log_BP}.}
    \label{fig:CH_log_mobility_date}
\end{figure}
From Figure \ref{fig:CH_log_mobility_date}, we observe that: 1. the maximum and minimum values of the computed solution agree well with those of the reference solution, and the numerical solution remains in $(-1,1)$;
2. by taking $\dt=10^{-4}$, the energy produced by the proposed method does not suffer from the non-physical oscillation, whereas the other two approaches produce oscillation at $t\approx 10^{-1}$ and $t\approx 1.3\times 10^{-1}$, where the standard FEM solution leaves the range $(-1,1)$; 
3. the iterative solver is only activated in a few time steps for $\dt=10^{-4}$, and never activated for smaller $\dt$, which means that the proposed method has almost the same computational cost of the standard approaches, with well-preserved fundamental physical properties;
4. mass-conservation is well preserved by the proposed scheme, up to accumulated rounding errors.
These observations demonstrate the stability and the efficiency of the proposed scheme, and highlight the role of the SAV technique in enhancing energy stability.

\subsection{Second-order parabolic problem}
In this section, we assess the accuracy of the proposed scheme \eqref{eq_scheme_second_order} first, then test it on a nonsmooth stationary case considered in \cite{Kirby_Shapero_2024,Barrenechea_Georgoulis_Pryer_2023}. Next, we consider the porous medium equations with analytic solution in the Barenblatt form \cite{Vazquez_2007}, known as a challenging test owing to the presence of vacuum regions and low regularity. 

\paragraph*{Accuracy test on smooth solution.} Here we consider the equation
\begin{equation}\label{eq_second_accuracy}
    \partial_t u = \DIV((1+u)\nabla u) + f_1,
\end{equation}
with the exact solution
\begin{equation}
    u(x,y,t)=\Big(\cos(x)\cos(y) - \frac{3}{8\pi}x^4 +x^3 \Big)\cos(t),
\end{equation}
and $f_1$ constructed based on $u$, in the domain $(0,2\pi)^2$ and $T=1$. The bounds of the above problem are given by $a(t)=-\cos(t),b(t)=(1+2\pi^3)\cos(t)$. In this problem, the Neumann boundary condition for $u$ is zero, but for $w\eqq-\sqrt{\epsilon}\Delta u$ is nonzero.

In Figure \ref{fig:accuracy_parabolic}, we illustrate the convergence test with $\dt=\frac{h}{2}$,  $p\in\{1,2,3,4\}$ and $k=p+1$, \ie we consider BDF$p+1$ in time. In addition, we artificially impose the homogeneous Neumann boundary condition for $w_h$, although the corresponding exact boundary condition is nonzero.

\begin{figure}[htb]
    \centering
    \includegraphics[width=0.5\linewidth]{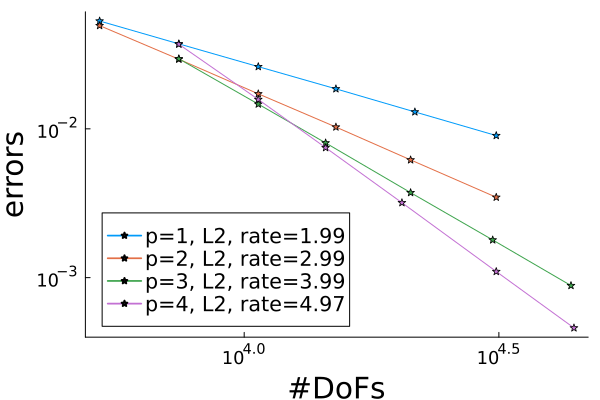}
    \caption{Accuracy test for the second-order parabolic equation \eqref{eq_second_accuracy}.}
    \label{fig:accuracy_parabolic}
\end{figure}

We observe from Figure \ref{fig:accuracy_parabolic} that the expected convergence rate in the $L^2$-norm at the final time for $p\in\{1,2,3,4\}$ are obtained. Thus, the numerical evidence indicates that the perturbation induced by the small fourth-order elliptic operator does not affect the observed convergence order, even though there is a small mismatch on the boundary condition for $w$.

\paragraph*{2D stationary positivity test.} We consider here a benchmark test used in \cite{Kirby_Shapero_2024}. The objective of this test is to compare the proposed method \eqref{eq_scheme_second_order} with the standard FEM. We consider the following equation
\begin{equation}\label{eq_second_BP}
    -\DIV( K(x,y)\nabla u ) =f_1, 
\end{equation}
with 
\begin{equation}
    K(x,y)=\begin{bmatrix}
        x^2+10^{-4}y^2 & -(1-10^{-4})xy \\
        -(1-10^{-4})xy & y^2+10^{-4}x^2 
        \end{bmatrix}, \qquad f_1(x,y)=\begin{cases}
            1, \text{ if } |x-\frac{1}{2}|\leq \frac{1}{8},\\
            0, \text{ otherwise,}
        \end{cases}
\end{equation}
and the homogeneous Dirichlet boundary conditions for $u$. Since $f_1$ is nonnegative, the solution $u$ is expected to be nonnegative. To solve such PDE, we perturb it by adding $\epsilon\Delta^2 u$ in the PDE to get
\begin{equation*}
    -\DIV( K(x,y)\nabla u ) -\sqrt{\epsilon} \Delta w= f_1,\qquad w= \sqrt{\epsilon} \Delta u, 
\end{equation*}
with the homogeneous Dirichlet boundary conditions for $u$ and $w$, and $\epsilon=h^{p+1}$. The perturbed problem is then discretized using the proposed scheme \eqref{eq_scheme_second_order}. The Dirichlet boundary conditions are strongly imposed by replacing $V_h^p$ and $V_h^{p,+}$ in the scheme by their counterpart with homogeneous Dirichlet boundary conditions.

We compare the proposed scheme with the standard FEM approximation. We construct three numerical solutions here for comparison: the reference solution $u_{h,ref}$ is constructed by standard $\mathbb{P}_1$-FEM with $200$ nodes in each direction; the FEM solution $u_{h,FEM}$ is delivered by $\mathbb{P}_1$-FEM with $15$ nodes in each direction; the solution $u_{h,VI}$ is delivered by our scheme with $p=1$ and $15$ nodes in each direction.
In Figure \ref{fig:laplace_2d_osc},  we illustrate the reference solution $u_{h,ref}$ in the left panel; in the middle panel and right panel, we present the pointwise numerical error $u_{h_ref}-u_{h,FEM}$ and $u_{h_ref}-u_{h,VI}$, respectively. 

\begin{figure}[htb]
    \begin{subfigure}{0.3\textwidth}
        \includegraphics[width=\textwidth]{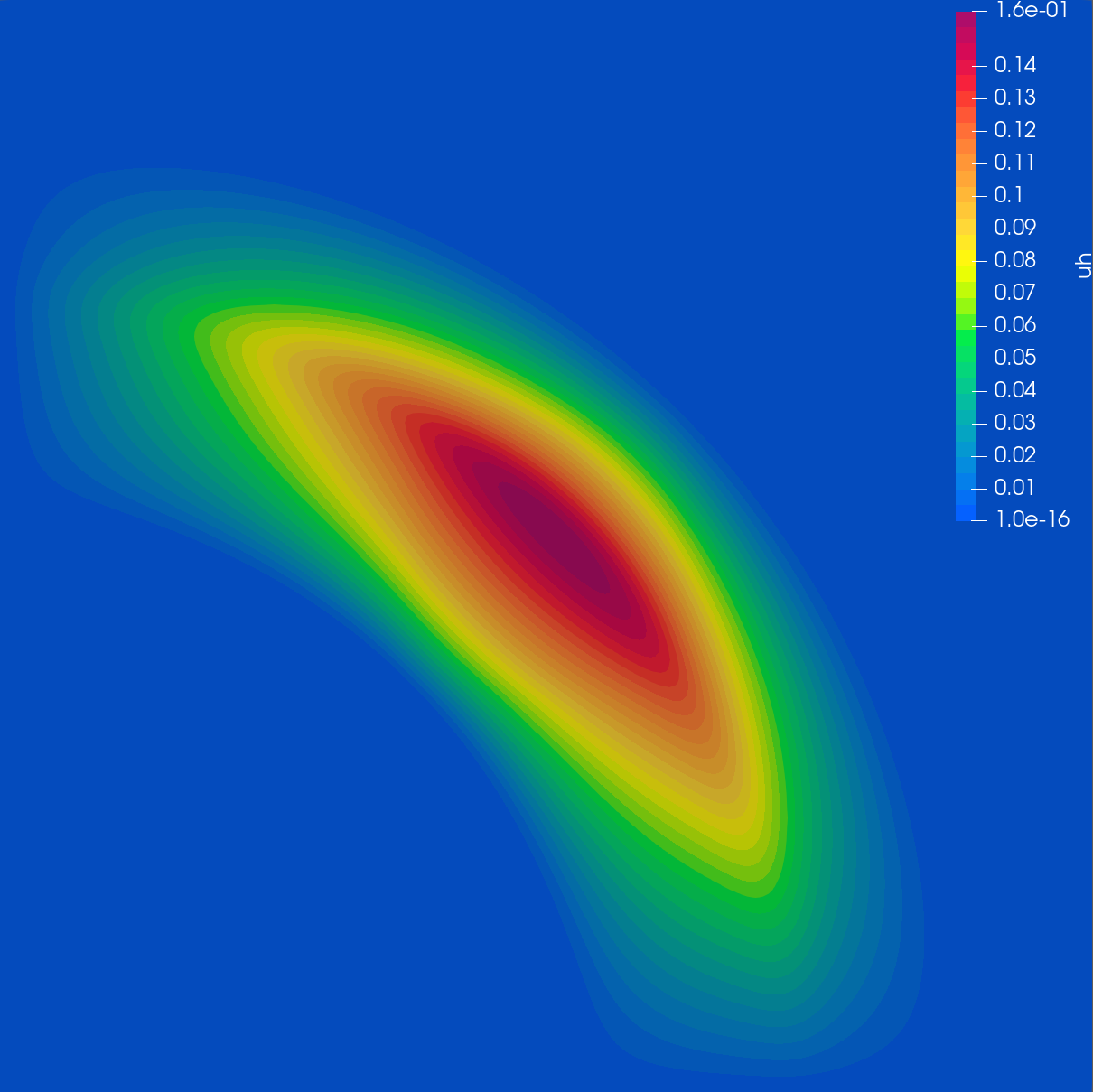}
        \caption{Reference solution.}
    \end{subfigure}
    \begin{subfigure}{0.3\textwidth}
        \includegraphics[width=\textwidth]{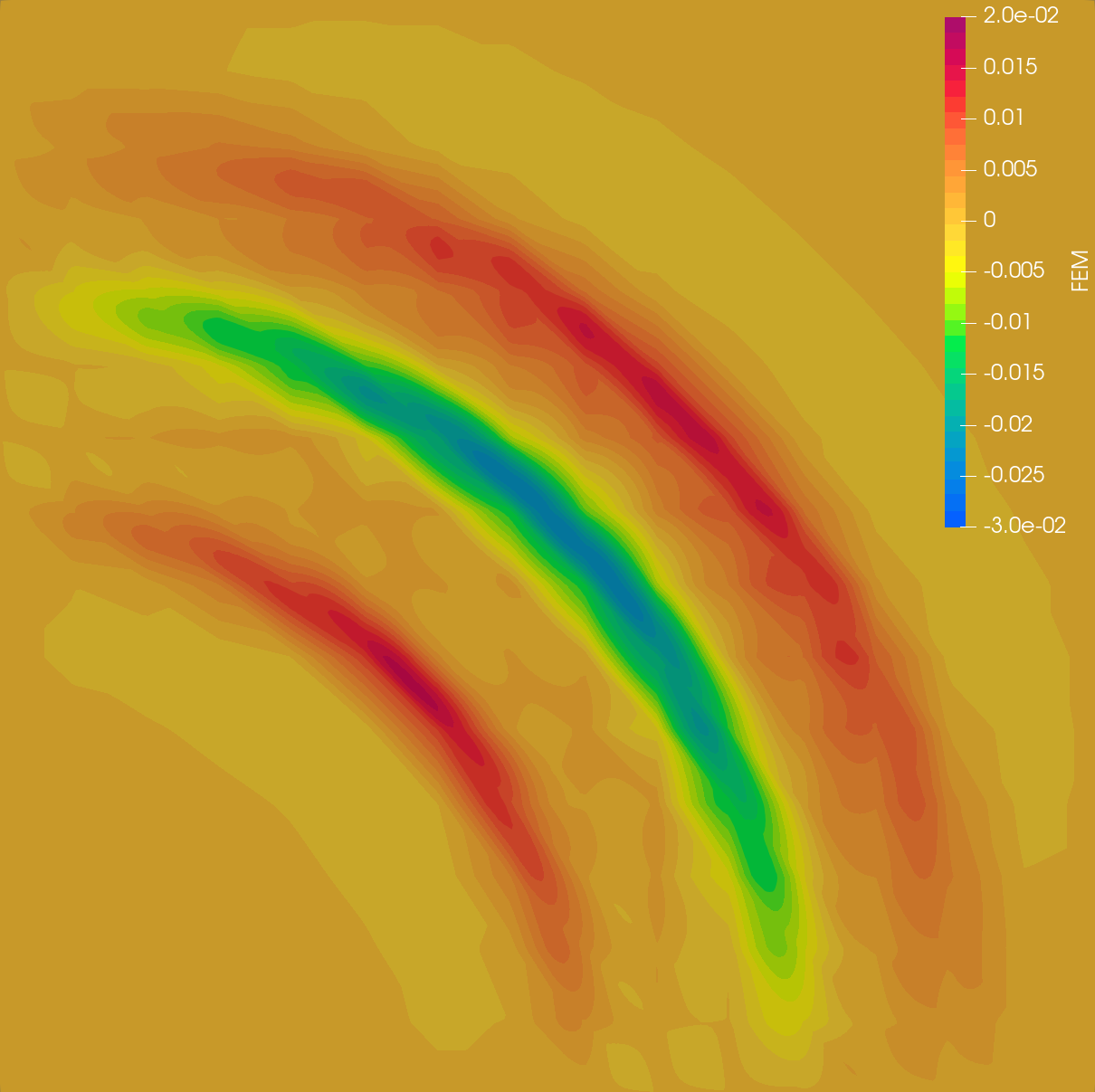}
        \caption{Error for standard FEM.}
    \end{subfigure}
    \begin{subfigure}{0.3\textwidth}
        \includegraphics[width=\textwidth]{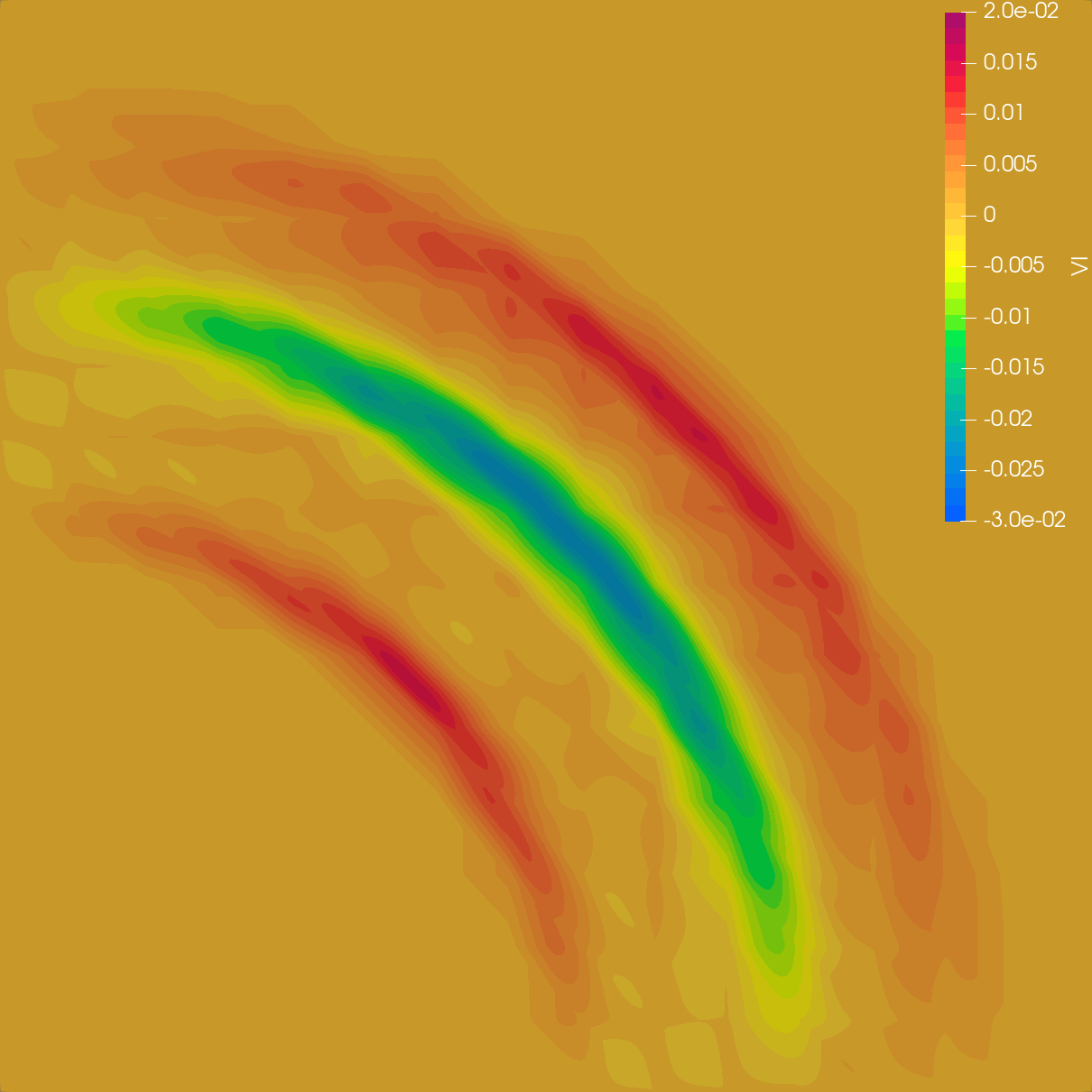}
        \caption{Error for our scheme.}
    \end{subfigure}
    \caption{Illustration of the numerical solutions for stationary second-order parabolic problem \eqref{eq_second_BP}.}
    \label{fig:laplace_2d_osc}
\end{figure}

From Figure \ref{fig:laplace_2d_osc}, we observe that the proposed method successfully removes the negative part of the numerical solution, which indicates the advantage of the proposed scheme compared with the standard FEM approximation. This behavior is similar to that observed for other variational inequality based FEM schemes in \cite{Kirby_Shapero_2024,Barrenechea_Georgoulis_Pryer_2023}.

\paragraph{Nonsmooth test on porous medium equation.} Here we consider a singular test \cite{Vazquez_2007} to illustrate the stability and accuracy of the proposed method:
\begin{equation}\label{eq_porous_medium}
    \partial_t u = m \DIV( u^{m-1} \nabla u ),
\end{equation}
in the domain $(-5,5)$, $m\geq 1$, $T=1.2$, with homogeneous Neumann boundary conditions. The exact solution of the porous medium equation in the Barenblatt form is
\begin{equation}
    u(x,t)=\frac{1}{t_0^\alpha} \max(0, 1-\alpha \frac{m-1}{2m}\frac{x^2}{t_0^{2\alpha}})^\frac{1}{m-1},
\end{equation}
where $t_0=t+1$ and $\alpha=\frac{1}{m+1}$. This solution is compactly supported with the interface moving outward in a finite speed. The main challenges in the numerical simulation of this equation are the limited regularity of the solution and the requirement of positivity. Indeed, for all $t\in [0,T]$, $u(\cdot,t)\in H^2(\Dom)$ for $m \leq 2$; $u(\cdot,t)\in H^1(\Dom)$ for $m < 3$; while for $m>3$, the solution $u(\cdot,t)\in H^s(\Dom)$, $s<\frac{1}{2}+\frac{1}{m-1}$. In addition, the diffusion term may be negative when $u_h<0$. 

In Figure \ref{fig:porous_medium}, we assess accuracy, structure-preserving properties, and robustness for the porous medium equation, and compare the proposed scheme with representative methods from the literature.
We discretize the PDE with $p=1$, $k=2$ and $\dt=\frac{h}{10}$, on uniform meshes. In the left panel, we illustrate the numerical solution delivered by the minimization postprocessing \cite{Liu_Riviere_Shen_Zhang_2024}, the bound-preserving and mass-conservative Lagrange multiplier method \cite{Cheng_Shen_2022_1}, a bound-preserving method \cite{Amiri_Barrenechea_Pryer_2025} and our method, marked as 'optimization', 'Lagrange multiplier', 'bound-preserving' and 'variational inequality' in the figure, respectively, at the final time, with $m=6$ and $400$ nodes in space; in the middle panel, we compare the convergence behavior of the bound-preserving scheme, the bound-preserving and mass-conservative Lagrange multiplier method and our method for $m=6$, in the $L^2$-norm at the final time; in the right panel, the convergence behavior of the proposed method is tested for $m\in\{2,3,4,5,6\}$ in the $L^2$-norm at the final time.

\begin{figure}
    \centering
    \begin{subfigure}{0.3\textwidth}
        \includegraphics[width=\textwidth]{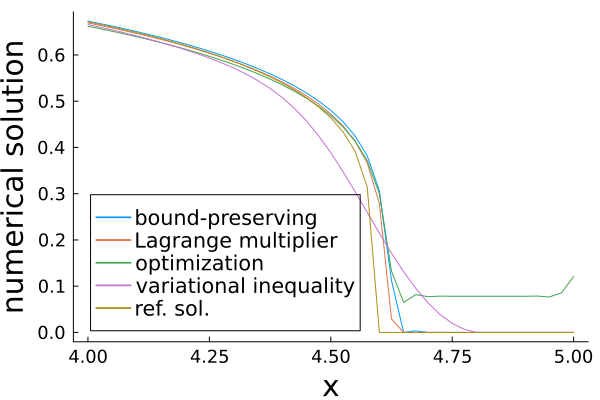}
        \caption{Numerical solution produced by different methods, $m=6$.}
    \end{subfigure}
    \begin{subfigure}{0.3\textwidth}
        \includegraphics[width=\textwidth]{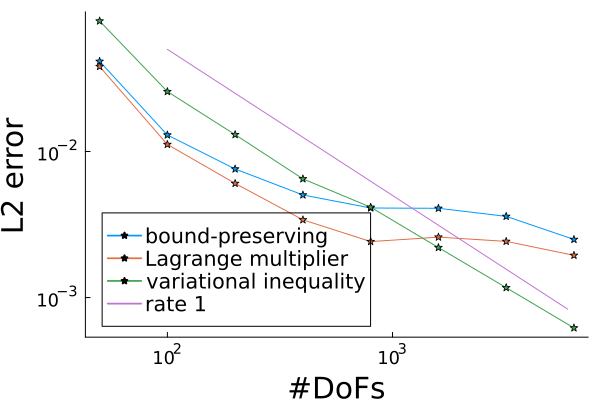}
        \caption{Convergence test for different methods, $m=6$.}
    \end{subfigure}
    \begin{subfigure}{0.3\textwidth}
        \includegraphics[width=\textwidth]{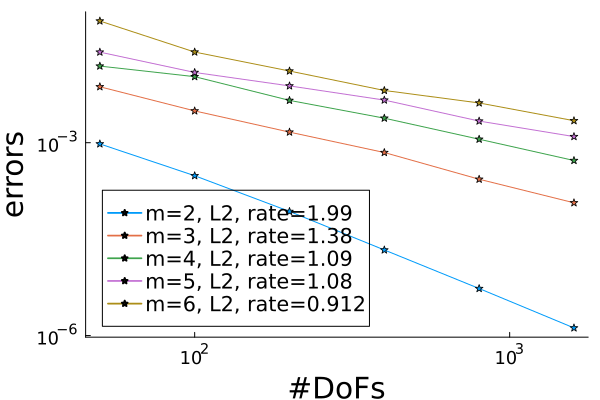}
        \caption{Convergence test for the proposed method.}
    \end{subfigure}

    \caption{Numerical illustration for the porous medium equation in the Barenblatt form \eqref{eq_porous_medium}.}
    \label{fig:porous_medium}
\end{figure}

We have the following observations from Figure \ref{fig:porous_medium}: 1. for $m=6$ and $\dt=\frac{h}{10}=2.5\times 10^{-3}$, the optimization-based approach fails to produce an accurate solution under this discretization setting, the proposed method produces a mollified solution satisfying the positivity constraint, the Lagrange multiplier technique and bound-preserving method produce the most accurate solutions, which suggests using these two methods in this discrete setting; 2. although Lagrange multiplier and bound-preserving methods deliver a relatively more accurate solution with coarse meshes and large time step, they fail to exhibit linear convergence when we continue to refine mesh and reduce time step, for $m=6$. This suggests using the proposed method, which exhibits linear convergence across all tested mesh sizes and time steps; 3. The proposed method has second-order convergence for $m=2$ and first-order convergence for $m\geq 3$. Moreover, increasing $m$ leads to larger errors under the same discrete setting, which is consistent with the fact that the solution becomes increasingly singular and harder to resolve as $m$ increases. Overall, the proposed scheme demonstrates competitive performance compared with existing methods, particularly in maintaining stability and convergence across challenging regimes.

\section{Concluding remarks}
\label{sec:conclusions}

In this work, we developed a variational inequality-based framework for a broad class of fourth-order elliptic and parabolic problems. For linear fourth-order elliptic equations, the proposed approach yields finite element schemes that are bound-preserving and mass-conservative, and that admit an equivalent strictly convex minimization formulation. This structure allows us to establish well-posedness and derive an optimal error estimate in the $H^1$-seminorm for sufficiently smooth solutions. It also enables an efficient implementation via primal--dual active set methods.

We then extended the framework to nonlinear fourth-order parabolic problems by coupling the variational inequality formulation with BDF temporal discretizations and SAV stabilization. The resulting schemes are designed to achieve high-order accuracy in both space and time, while rigorously preserving bounds and mass. In addition, its first-order version is energy stable with a modified energy. Numerical experiments confirm that these schemes are robust and efficient on several challenging benchmark problems, including lubrication-type equations and Cahn--Hilliard models with nonlinear mobility and logarithmic potentials.

As a further application, we proposed a consistent fourth-order regularization strategy for nonlinear second-order parabolic problems and applied the same variational inequality--BDF framework. The resulting method enforces bound-preservation and mass-conservation. Extensive numerical tests, including to problems with low regularity and moving interfaces such as porous medium equations in Barenblatt form, indicated that they were highly stable and accurate in practice. However, a complete well-posedness theory for this regularized second-order setting remains an open problem.

Compared with existing approaches, the present framework addresses several gaps in the literature:
\begin{itemize}
  \item it provides high-order spatial schemes for fourth-order elliptic problems that are simultaneously nodally bound-preserving and mass-conservative, with a rigorous $H^1$-error analysis;
  \item it yields space--time high-order, bound-preservation and mass-conservation for nonlinear fourth-order parabolic problems, and the energy stability is established for the first-order temporal discretization;
  \item it offers a practical high-order strategy for nonlinear second-order parabolic problems that enforces bounds and mass-conservation.
\end{itemize}

  Error analysis in the $L^2$-norm for the elliptic scheme \eqref{eq_scheme_stationary}, rigorous convergence analysis for parabolic schemes \eqref{eq_scheme}, \eqref{eq_scheme_without_f} and \eqref{eq_scheme_second_order}, as well as extensions to more complex systems such as coupled multiphysics models and higher-order PDEs, will be considered in future work.

\paragraph{Acknowledgements.}
This work is supported in part by National Natural Science Foundation of China grants 12371409 and 12501555.

\bibliographystyle{abbrvnat} 
\bibliography{bib}

\end{document}